\documentclass[english]{myarticle}

\title{Center Manifolds and Normal Forms for Nonlinearly Periodically Forced DDEs}

\author{Bram Lentjes\thanks{Department of Mathematics, Hasselt University, Diepenbeek Campus, 3590 Diepenbeek, Belgium \email{(bram.lentjes@uhasselt.be)}.}
\and Seppe Dani\"els\thanks{Department of Mathematics, KU Leuven, 3000 Leuven, Belgium  \email{(seppe.daniels@student.kuleuven.be)}.}
\and Meinder Follon\thanks{Department of Mathematics, TU Delft, 2600 AA Delft, The Netherlands \email{(m.follon@student.tudelft.nl)}.}
\and Yuri A. Kuznetsov\thanks{Department of Mathematics, Utrecht University, 3508 TA Utrecht, The Netherlands and Department of Applied
Mathematics, University of Twente, 7500 AE Enschede, The Netherlands \email{(i.a.kouznetsov@uu.nl)}.}}

\DeclareMathOperator{\BC}{BC}
\DeclareMathOperator{\NBV}{NBV}
\DeclareMathOperator{\loc}{loc}
\DeclareMathOperator{\inv}{INV}

\begin{document}
\maketitle

\begin{abstract}
The aim of this paper is to provide an effective framework for analysing bifurcations of equilibria in nonlinearly periodically forced delay differential equations. First, we establish the existence of a periodic smooth finite-dimensional center manifold near a nonhyperbolic equilibrium using the rigorous functional analytic framework of dual semigroups (sun-star calculus). Second, we construct a center manifold parametrization that allows us to describe the local dynamics on the center manifold near the equilibrium in terms of periodically forced normal forms. Third, we present a normalization method to derive explicit computational formulas for the critical normal form coefficients at a bifurcation of interest. In particular, we obtain such formulas for the periodically forced fold and nonresonant Hopf bifurcation. Several examples and indications from the literature confirm the validity and effectiveness of our approach.
\end{abstract}


\begin{sloppypar}


\section{Introduction} \label{sec:introduction}
Consider the autonomous \emph{classical delay differential equation} 
\begin{equation} \label{eq:IntroDDE}
    \dot{x}(t) = f(x_t), \quad t \geq 0,
\end{equation}
where the unknown $x$ takes values in $\mathbb{R}^n$, and the \emph{history} of $x$ at time $t$ is defined by $x_t(\theta) \coloneqq x(t+\theta)$ for all $\theta \in [-h,0]$. We take the state space $X$ to be the Banach space $C([-h,0], \mathbb{R}^n)$ consisting of all $\mathbb{R}^n$-valued continuous functions defined on $[-h,0]$ equipped with the supremum norm. Here, $0 < h < \infty$ denotes the finite maximal delay and the (nonlinear) function $f : X \to \mathbb{R}^n$ is assumed to be $C^{k}$-smooth for some sufficiently large $k \geq 2$. 

To study solutions of \eqref{eq:IntroDDE}, it is natural to construct a dynamical systems framework analogous to that used for (in)finite-dimensional ordinary differential equations \cite{Kuznetsov2023a,Haragus2011,Guckenheimer1983,Arnold1988}. This involves the development of essential nonlinear tools, such as center manifolds and normal forms near equilibria and cycles, along with a normalization technique that enables the direct computation of normal form coefficients at a bifurcation of interest. These normal form coefficients determine the nature and direction (degenerate, sub- or supercritical) of a bifurcation and provide detailed information on the dynamics near a nonhyperbolic equilibrium or cycle. It turns out that such results have been established for \eqref{eq:IntroDDE} through the application of the rigorous dual perturbation framework (called \emph{sun-star calculus}) developed in \cite{Clement1987,Clement1988,Clement1989,Clement1989a,Diekmann1991}, see also the well-known book \cite{Diekmann1995} for an introduction. Below, we provide a brief overview of the (recent) literature available on this subject.

First, the existence of a $C^{k}$-smooth finite-dimensional center manifold $\mathcal{W}_{\loc}^c(\overline{x}) \subseteq X$ near a nonhyperbolic equilibrium $\overline{x}$ for \eqref{eq:IntroDDE} has been rigorously established, see \cite{Diekmann1991CM,Diekmann1995} for the critical center manifold and \cite{Bosschaert2020} for the parameter-dependent case. Restricting \eqref{eq:IntroDDE} to $\mathcal{W}_{\loc}^c(\overline{x})$ reduces it to a finite-dimensional ODE, allowing the use of classical ODE normal form theory to analyse bifurcations of equilibria. This allowed researchers to derive explicit computational formulas for the normal form coefficients of all codim 1 and 2 bifurcations of equilibria using a normalization technique, see \cite{Bosschaert2020,Diekmann1995,Janssens2010,Bosschaert2024a} for further details, and \cite{Faria1995,Faria1995a,Faria2006} for another approach. These formulas have been implemented in the \verb|MATLAB| package \verb|DDE-BifTool| \cite{Engelborghs2002,Sieber2014} and the \verb|Julia| package \verb|BifurcationKit| \cite{Veltz2020}, and are widely used to analyse equilibrium bifurcations in a broad variety of mathematical models arising across the sciences.

Second, the existence of a periodic $C^{k-1}$-smooth finite-dimensional center manifold $\mathcal{W}_{\loc}^c(\Gamma) \subseteq X$ near a nonhyperbolic cycle $\Gamma$ for \eqref{eq:IntroDDE} has been recently rigorously established in \cite{Article1}. Periodic normal forms are derived in terms of a natural $(\tau,\xi)$-coordinate system on $\mathcal{W}_{\loc}^c(\Gamma)$, see \cite{Article2} for a primary reference and \cite{Iooss1988,Iooss1999,Kuznetsov2023a} for further details. This framework enabled the authors from \cite{Bosschaert2025} to derive explicit computational formulas for all codim 1 bifurcations of limit cycles, which have been implemented in the \verb|Julia| package \verb|PeriodicNormalizationDDEs| \cite{Bosschaert2024c}, that is compatible with \verb|BifurcationKit|. Similar formulas for all codim 2 bifurcations will be derived and incorporated into a software package like \verb|PeriodicNormalizationDDEs| or \verb|DDE-Biftool| in an upcoming publication, employing the periodic normalization method described in \cite{Witte2013,Witte2014,Kuznetsov2005}.

The next obvious step is to use a similar powerful scheme (center manifolds, normal forms and normalization) for nonautonomous classical DDEs of the form
\begin{equation} \tag{DDE} \label{eq:DDE}
    \dot{x}(t) = F(t,x_t), \quad t \geq s,
\end{equation}
where $s \in \mathbb{R}$ denotes a starting time. Similarly to \eqref{eq:IntroDDE}, we assume that $X$ is our state space and the (non)linear map $F : \mathbb{R} \times X \to \mathbb{R}^n$ is $C^{k}$-smooth. Moreover, we assume that \eqref{eq:DDE} admits a constant (equilibrium) solution $\overline{x}$. It is well known that the bifurcation theory of  equilibrium solutions for general nonautonomous ODEs, let alone for the more challenging case \eqref{eq:DDE}, remains underdeveloped and is currently an active area of research \cite{kloeden2020,Anagnostopoulou2023,Rasmussen2007,Rasmussen2007a,Caraballo2016}. Therefore, we will restrict ourselves to a special case which arises quite often in applications, namely that of  periodically forced DDEs with purely nonlinear driving force that we call \emph{nonlinearly periodically forced DDEs}. Such systems exhibit linear autonomous behaviour, but are nonlinearly periodically driven by an external source/force. To exploit this structure, it is convenient to translate $\overline{x}$ towards the origin via $y = x -\overline{x}$ and to study solutions of
\begin{equation} \label{eq:introDDE3}
    \dot{y}(t) = Ly_t + G(t,y_t), \quad t \geq s,
\end{equation}
near the origin, where the linear part $L \coloneqq D_2F(t,\overline{x}) \in \mathcal{L}(X,\mathbb{R}^n)$ is assumed to be time-independent and the nonlinear part $G(t,\varphi) \coloneqq F(t,\overline{x} + \varphi) - F(t,\overline{x}) - L\varphi$ is assumed to be $T$-periodic in the first argument. Note that \eqref{eq:introDDE3} also naturally emerges as a mixture of the two described above scenarios: the operator $L$ represents the autonomous linear part from equilibria, while the nonautonomous nonlinearity $G$ captures the periodic component from limit cycles. This paper aims to develop a center manifold and normal form theory, as well as a normalization method for bifurcations of equilibria in the nonlinearly periodically forced system \eqref{eq:introDDE3}. 

Our approach is primarily inspired by the work of Iooss and collaborators \cite{Haragus2011,Elphick1987} on nonlinearly periodically forced ODEs. Real-world applications of nonlinearly periodically forced systems naturally occur in the analysis of (auto)parametric resonance in mechanical systems and oscillators, a phenomenon with numerous engineering applications. For an introduction to this topic, we refer the reader to \cite{Tondl2000,Verhulst2002,Verhulst2009,Hoveijn1995} and the references therein.

\subsection{Challenges and overview}
The paper is organised as follows. In \Cref{sec:CMtheorem}, we prove the existence of a $T$-periodic $C^k$-smooth finite-dimensional center manifold $\mathcal{W}_{\loc}^c(\overline{x})$ near $\overline{x}$ for the nonlinearly periodically forced system \eqref{eq:DDE}. In contrast to the earlier constructed center manifolds for equilibria and cycles in $X$ for the autonomous DDE \eqref{eq:IntroDDE}, the center manifold $\mathcal{W}_{\loc}^c(\overline{x})$ for the nonautonomous system \eqref{eq:DDE} becomes a nonautonomous invariant subset of $\mathbb{R} \times X$. Our approach relies on the Lyapunov-Perron method, which we apply to a suitable variation of constants formula adapted to nonlinearly periodically forced systems within the sun-star calculus framework. This formulation enables us, under the relatively mild assumptions given in \Cref{hyp:1} and \Cref{hyp:2}, to prove a center manifold theorem that extends to a considerably broader class of nonlinearly periodically forced evolution equations, see \Cref{thm:LCMT} for the general result and \Cref{remark:LCMTextension} for some further extensions. In particular, this method also provides a direct proof on the existence of such center manifolds for nonlinearly periodically forced renewal equations \cite{Breda2020,Diekmann2008,Diekmann1995}. As a consequence, we recover a well-known center manifold result \cite[Theorem 3.9]{Haragus2011} from finite-dimensional nonlinearly periodically forced ODEs as discussed in \Cref{remark:CMODE}.

In \Cref{sec:NFtheorem}, we show that the local dynamics near the equilibrium $\overline{x}$ of \eqref{eq:DDE} on $\mathcal{W}_{\loc}^c(\overline{x})$ can be analysed in terms of \emph{periodically forced normal forms}. We extend the results of Iooss et al. \cite{Haragus2011,Elphick1987} on periodically forced normal forms from ODEs to DDEs. To achieve this, we parametrize any solution $x_t$ of \eqref{eq:DDE} on $\mathcal{W}_{\loc}^c(\overline{x})$ as
\begin{equation} \label{eq:parametrization}
    x_t = \overline{x} + Q_0 \xi + H(t,\xi),
\end{equation}
where $\xi \in \mathbb{R}^{n_0}$ represents a normal form coordinate, $Q_0 : \mathbb{R}^{n_0} \to X_0$ is a coordinate map capturing the linear component in the $n_0$-dimensional center subspace $X_0$, and $H : \mathbb{R} \times \mathbb{R}^{n_0} \to X$ represents the nonlinear part of the center manifold expressed in terms of the $(t,\xi)$-coordinate system. The normal form coordinate $\xi$ evolves according to a specific nonlinearly periodically forced ODE in $\mathbb{R}^{n_0}$, where the linear part is given by the associated Jordan matrix associated to $X_0$ and the nonlinear part can be computed explicitly up to order $k$. We have found that the (critical) periodically forced normal forms for \eqref{eq:DDE} are precisely identical to those for ODEs derived in \cite[Theorem 3.5.2]{Haragus2011}. The main result is stated in \Cref{thm:periodicnormalform}. This theorem can be regarded as a periodically forced version of \cite[Theorems 24-26]{Article2}, where periodic normal forms for bifurcations of limit cycles for \eqref{eq:IntroDDE} are derived. An alternative method for deriving periodically forced normal forms for \eqref{eq:DDE} is proposed in the work of Faria \cite{Faria1997}, which builds on the formal adjoint approach for DDEs developed in \cite{Hale1993}. However, the main result in \cite[Theorem 6.4]{Faria1997} relies on the assumption of the existence of a periodic center manifold $\mathcal{W}_{\loc}^c(\overline{x})$, which had not yet been rigorously established at that time. This gap is now closed due to the center manifold result from \Cref{sec:DDEs}. 

In \Cref{sec:spectral}, we first recall the explicit expressions for the (adjoint) (generalized) eigenfunctions for autonomous linear DDEs, which are expressed in terms of the Jordan chains of the characteristic matrix $\Delta(z) \in \mathbb{C}^{n \times n}$. In the next section, where explicit computational formulas for the critical normal form coefficients are derived, we will encounter a specific class of periodic linear operator equations. Solutions to these equations can be described in terms of a characteristic operator $\boldsymbol{\Delta}(z) : C_T^1(\mathbb{R}, \mathbb{C}^n) \subseteq C_T(\mathbb{R}, \mathbb{C}^n) \to C_T(\mathbb{R}, \mathbb{C}^n)$, which naturally extends $\Delta(z)$ in a $T$-periodic fashion. Despite acting on an infinite-dimensional function space, the key advantage for the analysis is that the (inhomogeneous) periodic linear DDE $\boldsymbol{\Delta}(z)\boldsymbol{q} = \boldsymbol{f}$ is equivalent to the much simpler algebraic system $\Delta(z + im\omega_T)q_m = f_m$ for all $m \in \mathbb{Z}$ in Fourier space. Here, $q_m$ and $f_m$ denote the Fourier coefficients of $\boldsymbol{q}$ and $\boldsymbol{f}$ respectively, and $\omega_T \coloneqq 2\pi/T$ is the angular frequency. Next, we discuss solvability of these equations and invoke a suited Fredholm alternative that will be used extensively in the next section.

In \Cref{sec:normalization}, we derive explicit computational formulas for the critical normal form coefficients corresponding to the periodically forced fold \eqref{eq:criticalcoeffffold} and nonresonant Hopf \eqref{eq:criticalcoeffHopf} bifurcation in nonlinearly periodically forced DDEs. This is accomplished by applying the center manifold parametrization \eqref{eq:parametrization} together with the homological equation \eqref{eq:homological}. The expressions for these normal form coefficients are explicit, ready-to-implement in a software package and straightforward to evaluate (numerically). Furthermore, they reduce naturally to the classical formulas for the critical normal form coefficients for the fold and Hopf bifurcation in autonomous systems ($T = 0$), see \Cref{remark:fold} and \Cref{remark:Hopf} for further details. The presented approach is broadly applicable and extends naturally to the computation of normal form coefficients for bifurcations of equilibria of higher codimension. 

In \Cref{sec:Examples}, we illustrate our results using two rather simple systems. First, we analyse a periodically forced fold bifurcation in a scalar model equation by explicitly computing the associated critical normal form coefficient. Second, we study the nonlinearly periodically forced Wright equation and compute the first Lyapunov coefficient $l_1$ for nonresonant Hopf bifurcations. This example demonstrates how periodic forcing influences the classical bifurcation picture of Wright's equation, leading to resonance phenomena, changes in criticality, and degenerate bifurcation scenarios that do not arise in the autonomous setting. Compared with the derivation of $l_1$ by Faria \cite[Section 7]{Faria1997}, our method is significantly simpler and does not require solving additional boundary-value problems.

In \Cref{sec:outlook}, we conclude the paper with a summary of our main contributions and a brief discussion of the parameter-dependent setting. In particular, we outline directions for future software development, including the incorporation of our methods into existing numerical bifurcation tools.

\section{Periodically forced center manifold theorem} \label{sec:CMtheorem}
As already mentioned in \Cref{sec:introduction}, our first goal is to prove an abstract center manifold result (\Cref{thm:LCMT}) for nonlinearly periodically forced systems within the general framework of dual perturbation theory. To establish this result, some preliminary work is needed. In \Cref{subsec:functionalanalytic}, we give a brief overview of the sun-star calculus framework, particularly focusing on (non)linear perturbations. Standard references for this part include the classic works \cite{Diekmann1995,Neerven1992} and the sun-star articles mentioned in \Cref{sec:introduction}. In \Cref{label:spectraldecom}, we turn our attention to the (semi)linear theory and formulate \Cref{hyp:1} and \Cref{hyp:2}, which are crucial for stating our main theorem. In \Cref{subsec:modification}-\Cref{sec:properties}, we address the nonlinear aspects, like proving the center manifold theorem and deriving key properties of this manifold. In \Cref{sec:DDEs}, we apply this abstract result to the specific case of the nonlinearly periodically forced system \eqref{eq:DDE} for which we finally obtain a center manifold $\mathcal{W}_{\loc}^c(\overline{x})$ around the equilibrium $\overline{x}$, meaning that $\mathbb{R} \times \{\overline{x}\} \subseteq \mathcal{W}_{\loc}^c(\overline{x})$.

To keep this technical section concise, we have not aimed to completely rewrite \cite[Chapter IX]{Diekmann1995} (or \cite[Section 5-6]{Janssens2020}) and \cite[Section 2-3]{Article1} on center manifolds near equilibria and cycles within the sun-star calculus framework. Instead, we focus on proving the essential results, while referring to these references for proofs that carry over with minimal adjustments in our setting. All unreferenced claims related to basic properties of sun-star calculus (for DDEs) can be found in the mentioned references.

\subsection{Functional analytic framework} \label{subsec:functionalanalytic}
Let $T_0 := \{T_0(t)\}_{t \geq 0}$ be a strongly continuous semigroup of bounded linear operators ($\mathcal{C}_0$-semigroup) on a real or complex Banach space $X$, with (infinitesimal) generator $A_0$ and domain $\mathcal{D}(A_0)$. The \emph{dual semigroup} $T_0^\star := \{T_0^\star(t)\}_{t \geq 0}$, where each $T_0^\star(t): X^\star \to X^\star$ is the adjoint of $T_0(t)$, forms a semigroup on the topological dual space $X^\star$ of $X$. If $X$ is not reflexive, $T_0^\star$ is generally only weak$^\star$ continuous on $X^\star$. This limitation is also reflected on the generator level since the adjoint $A_0^\star$ of $A_0$ is merely the weak$^\star$ generator of $T_0^\star$ and typically has a non-dense domain. To address this, one considers the maximal subspace of strong continuity
\begin{equation*}
    X^\odot := \left\{ x^\star \in X^\star : t \mapsto T_0^\star(t)x^\star \text{ is norm continuous on } [0, \infty) \right\},
\end{equation*}
which is a norm-closed $T_0^\star(t)$-invariant weak$^\star$ dense subspace of $X^\star$. Moreover, $X^\odot$ is also characterized by the norm closure of $\mathcal{D}(A_0^\star)$ in $X^\star$. The restriction of $T_0^\star$ to $X^\odot$ defines a $\mathcal{C}_0$-semigroup $T_0^\odot$ on $X^\odot$, whose generator $A_0^\odot$ equals the part of $A_0^\star$ in $X^\odot$:
\begin{equation*}
    \mathcal{D}(A_0^\odot) = \left\{ x^\odot \in \mathcal{D}(A_0^\star) : A_0^\star x^\odot \in X^\odot \right\}, \quad A_0^\odot x^\odot = A_0^\star x^\odot.
\end{equation*}
At this point, we have a $\mathcal{C}_0$-semigroup $T_0^\odot$ with generator $A_0^\odot$ on the Banach space $X^\odot$, mirroring the original structure on $X$. Repeating the construction, we obtain on the dual space $X^{\odot \star}$ the weak$^\star$ continuous adjoint semigroup $T_0^{\odot \star}$ with weak$^\star$ generator $A_0^{\odot \star}$. Restricting this semigroup to the maximal subspace of strong continuity $X^{\odot \odot}$ yields a $\mathcal{C}_0$-semigroup $T_0^{\odot \odot}$ with generator $A_0^{\odot \odot}$ coinciding with the part of $A_0^{\odot \star}$ in $X^{\odot \odot}$. Let $\langle \cdot,\cdot \rangle$ denote the natural duality pairing and introduce the canonical embedding $j: X \to X^{\odot \star}$ defined by
\begin{equation} \label{eq:j}
    \langle jx, x^\odot \rangle := \langle x^\odot, x \rangle, \quad \forall x \in X,\ x^\odot \in X^\odot,
\end{equation}
which maps $X$ into $X^{\odot \odot}$. If $j$ maps $X$ onto $X^{\odot \odot}$, then $X$ is said to be \emph{$\odot$-reflexive} with respect to $T_0$. In the study of classical DDEs (\Cref{sec:DDEs}), $\odot$-reflexivity with respect to $T_0$ holds, and it massively simplifies the upcoming analysis. In the absence of $\odot$-reflexivity, it becomes for example less clear under which conditions the weak$^\star$ integrals from \eqref{eq:LAIE} and \eqref{eq:NLAIE} are well-defined, see \cite{Clement1989a,Janssens2019,Janssens2020,Spek2020} and \cite[Chapter III.8]{Diekmann1995} for further details on the non-$\odot$-reflexive setting.

With the abstract duality framework in place, we now turn to the analysis of bounded linear perturbations. Let $B : X \to X^{\odot \star}$ be a bounded linear operator. Then, there exists a unique $\mathcal{C}_0$-semigroup $T := \{T(t)\}_{t \geq 0}$ on $X$ that satisfies the linear abstract integral equation
\begin{equation} \label{eq:LAIE}
    T(t)\varphi = T_0(t)\varphi + j^{-1} \int_0^t T_0^{\odot \star}(t - \tau) B T(\tau) \varphi  d\tau, \quad \varphi \in X,
\end{equation}
for $t \geq 0$, where the integral must be interpreted as a weak$^\star$ Riemann integral. Under the running assumption of $\odot$-reflexivity, it is known that this particular integral takes values in $j(X)$, meaning that \eqref{eq:LAIE} is well-defined. On $X^{\odot \star}$, the perturbation $B$ appears additively in the action of $A_0^{\odot \star}$ as
\begin{equation} \label{eq:Asunstar}
    \mathcal{D}(A^{\odot \star}) = \mathcal{D}(A_0^{\odot \star}), \quad A^{\odot \star} = A_0^{\odot \star} + B j^{-1}.
\end{equation}
The generator $A$ of the perturbed semigroup $T$ can be recovered as the part of $A^{\odot \star}$ in $X^{\odot \odot}$ as
\begin{equation*}
    \mathcal{D}(A) = \left\{ \varphi \in X : j\varphi \in \mathcal{D}(A_0^{\odot \star}) \text{ and } A_0^{\odot \star} j\varphi + B\varphi \in X^{\odot \odot} \right\}, \quad A\varphi = j^{-1} ( A_0^{\odot \star} j\varphi + B\varphi ),
\end{equation*}
and we observe that the perturbation $B$ naturally moves back into the domain of the generator $A$.

Let us now turn attention to periodic nonlinear perturbations. Let $R : \mathbb{R} \times X \to X^{\odot \star}$ be a $C^{k}$-smooth operator for some $k \geq 1$, which is $T$-periodic in the first component and satisfies
\begin{equation} \label{eq:Rassumptions}
    R(t,0) = 0, \quad D_2 R(t,0) = 0, \quad \forall t \in \mathbb{R}.
\end{equation}
Consider the time-dependent nonlinear abstract integral equation
\begin{equation} \label{eq:NLAIE}
    u(t) = T(t-s) \varphi + j^{-1} \int_s^t T^{\odot \star}(t-\tau) R(\tau,u(\tau)) d\tau, \quad \varphi \in X,
\end{equation}
for $t \geq s$, where $s \in \mathbb{R}$ is a starting time. Since the nonlinearity $R$ is $C^k$-smooth, recall from the mean value inequality in Banach spaces \cite[Corollary 3.2]{Coleman2012} that $R$ is locally Lipschitz in the second component. A standard contraction argument shows that for any $\varphi \in X$ there exists a unique (maximal) solution $u_\varphi$ of \eqref{eq:NLAIE} on some (maximal) interval $I_\varphi \coloneqq [s,t_\varphi)$ with $s < t_\varphi \leq \infty$. The family of all such solutions naturally defines a time-dependent semiflow $S : \mathcal{D}(S) \subseteq \mathbb{R}^2 \times X \to X$ defined by 
\begin{equation} \label{eq:Sprocess}
    \mathcal{D}(S) \coloneqq \{(t,s,\varphi) \in [s,\infty) \times \mathbb{R} \times X : t \in I_\varphi \}, \quad S(t,s,\varphi) \coloneqq u_\varphi(t).
\end{equation}
In the framework of nonautonomous dynamical systems, the function $S$ is also called a \emph{process}, see \cite{Kloeden2011,Anagnostopoulou2023} for more information on this topic. Since $0 \in X$ is an equilibrium (stationary point) of $S$, we know that $I_0 = [s,\infty)$ and $S(t,s,0) = 0$ for all $t \geq s$.

\subsection{Spectral decomposition and the linear (in)homogeneous equation} \label{label:spectraldecom}
The construction from the previous section is, in general, not sufficient to prove the existence of a periodic smooth finite-dimensional center manifold for \eqref{eq:NLAIE}. To guarantee the existence of such a manifold, additional assumptions are required. This necessity was already observed in \cite{Diekmann1991CM,Diekmann1995}, where the authors assumed the existence of a direct sum decomposition of $X^{\odot \star}$ supplemented by additional assumptions at the $\odot \star$-level. However, since the structure of $X^{\odot \star}$ depends on the specific evolution equation of interest, it is essential to compute $X^{\odot \star}$ and its associated $\odot\star$-tools explicitly in order to verify the required assumptions. For this reason, it is often more practical to formulate the relevant hypothesis directly in the original space $X$ and subsequently lift it to $X^{\odot \star}$ at the abstract level, an approach also emphasized in \cite{Janssens2020,LentjesCMODE}. Since our linear part is autonomous, we can work with \cite[H-II]{Janssens2020} as stated below.
\begin{hypothesis} \label{hyp:1}
The space $X$ and the $\mathcal{C}_0$-semigroup $T$ on $X$ have the following properties:
\begin{enumerate}
    \item $X$ admits a direct sum decomposition
    \begin{equation} \label{eq:decomposition X hyp}
        X = X_{-} \oplus X_0 \oplus X_{+},
    \end{equation}
    where each summand is closed.
    \item The subspaces $X_{i}$ are positively $T$-invariant for $i \in \{-,0,+\}$.
    \item $T$ extends to a $\mathcal{C}_0$-group on $X_i$ for $i \in \{0,+\}$.
    \item The decomposition \eqref{eq:decomposition X hyp} is an exponential trichotomy on $\mathbb{R}$, meaning that there exist $a < 0 < b$ such that for every $\varepsilon > 0$ there exists a $K_\varepsilon > 0$ such that
    \begin{alignat*}{2}
        \|T(t)\varphi\| &\leq K_\varepsilon e^{a(t-s)} \|\varphi\|, \quad& \forall t \geq s,\ \varphi \in X_{-}, \\
        \|T(t)\varphi\| &\leq K_\varepsilon e^{\varepsilon|t-s|} \|\varphi\|, \quad& \forall t \in \mathbb{R},\ \varphi \in X_{0}, \\
        \|T(t)\varphi\| &\leq K_\varepsilon e^{b(t-s)} \|\varphi\|, \quad& \forall t \leq s,\ \varphi \in X_{+}.
    \end{alignat*}
\end{enumerate}
We call $X_{-},X_{0}$ and $X_{+}$ the \emph{stable subspace}, \emph{center subspace} and \emph{unstable subspace}, respectively. \hfill $\lozenge$
\end{hypothesis}

We call the zero equilibrium of \eqref{eq:NLAIE} \emph{hyperbolic} when the center subspace $X_0$ is trivial, and \emph{nonhyperbolic} when $X_0$ is a nontrivial subspace of $X$. However, this distinction plays no role in the analysis that follows.

The procedure for lifting \Cref{hyp:1} towards $X^{\odot \star}$ is detailed in \cite[Appendix~A]{Janssens2020}, see also \cite[Appendix A.1]{Article1} for a $T$-periodic variant. A key aspect of this lifting process is that the standard spectral projections $P_i : X \to X$ onto the subspaces $X_i$, defined via the classical Dunford contour integral formula, can be lifted to corresponding projections $P_i^{\odot \star} : X^{\odot \star} \to X^{\odot \star}$ onto the subspaces $X_i^{\odot \star}$ for each $i \in \{-,0,+\}$. Moreover, we will also require \cite[H-III]{Janssens2020}, which we recall below.

\begin{hypothesis} \label{hyp:2}
The subspace $X_{i}^{\odot \star}$ is contained in $j(X_i)$ for $i \in \{0,+\}$. \hfill $\lozenge$
\end{hypothesis}

In practical applications, these hypotheses are typically verified through a spectral decomposition of the spectrum $\sigma(A)$ of the (complexification of the) generator $A$. For further details regarding the complexification procedure, we refer to \cite[Sections~III.7 and~IV.2]{Diekmann1995}. The following result applies to a broad class of $\mathcal{C}_0$-semigroups.

\begin{theorem}[{\cite[Theorem 28]{Janssens2020}}] \label{thm:spectralgab}
Suppose that $T$ is eventually norm continuous and let $\sigma(A)$ be the pairwise disjoint union of
\begin{equation*}
    \sigma_{-} \coloneqq \{ \lambda \in \sigma(A) : \Re \lambda < 0\}, \quad \sigma_{0} \coloneqq \{ \lambda \in \sigma(A) : \Re \lambda = 0\}, \quad \sigma_{+} \coloneqq \{ \lambda \in \sigma(A) : \Re \lambda > 0\}, \quad
\end{equation*}
where $\sigma_{-}$ is closed, while $\sigma_0$ and $\sigma_{+}$ are compact. If $\sup_{\lambda \in \sigma_{-}} \Re \lambda < 0 < \inf_{\lambda \in \sigma_{+}} \Re \lambda$, then \emph{\Cref{hyp:1}} and \emph{\Cref{hyp:2}} hold for $T$ on $X$.
\end{theorem}
To construct a center manifold for \eqref{eq:NLAIE}, we will be interested in solutions to this equation that exist for all time. It is therefore helpful to write this equation into translation invariant form
\begin{equation} \label{eq:inhomogenous}
    u(t) = T(t-s)u(s) + j^{-1}  \int_s^t T^{\odot \star}(t-\tau)R(\tau,u(\tau)) d\tau, \quad - \infty < s \leq t < \infty.
\end{equation}
A key challenge in developing a center manifold theory for infinite-dimensional systems is that the linearized version of \eqref{eq:inhomogenous} may admit unbounded solutions in $X_0$. This necessitates working within a function space that permits limited exponential growth as $t \to \pm \infty$. Therefore, let us introduce for a Banach space $E$, exponent $\eta \in \mathbb{R}$ and starting time $s \in \mathbb{R}$, the Banach space
\begin{equation*}
    \BC_{s}^{\eta}(\mathbb{R},E) := \bigg \{ f \in C(\mathbb{R},E) : \sup_{t \in \mathbb{R}} e^{-\eta|t-s|}\|f(t)\| < \infty \bigg \},
\end{equation*}
with the weighted supremum norm $\|f\|_{\eta,s} := \sup_{t \in \mathbb{R}} e^{-\eta|t-s|}\|f(t)\|$. Before we start studying the inhomogeneous equation \eqref{eq:inhomogenous}, let us first derive some properties of the linear homogeneous equation
\begin{equation} \label{eq:homogeneous}
    u(t) = T(t-s)u(s), \quad -\infty < s \leq t < \infty.
\end{equation}
The following result is a $s$-dependent version of \cite[Lemma IX.2.4]{Diekmann1995} (or \cite[Lemma 29]{Janssens2020}) and shows that all solutions to \eqref{eq:homogeneous} in $X_0$ belong to $\BC_{s}^{\eta}$. As the proof closely follows that of the cited references, it is omitted here. We also refer to \cite[Proposition 2]{Article1} for a similar result in the $T$-periodic setting.

\begin{lemma} \label{lemma:X0}
Let $\eta \in (0,\min\{-a,b\})$ and $s \in \mathbb{R}$ be given. Then
\begin{equation*}
    X_0 = \{ \varphi \in X : \mbox{there exists a solution of \eqref{eq:homogeneous} on $\mathbb{R}$ through $\varphi$ at time $s$ belonging to } \BC_{s}^{\eta}(\mathbb{R},X) \}.
\end{equation*}
\end{lemma}

In the construction of a center manifold, an important step is the introduction of a bounded linear operator that associates with each continuous forcing function $f : \mathbb{R} \to X^{\odot \star}$ a solution of 
\begin{equation} \label{eq:inhomogeneous CMT}
    u(t) = T(t-s)u(s) + j^{-1} \int_s^t T^{\odot \star}(t-\tau)f(\tau) d\tau, \quad -\infty < s \leq t < \infty,
\end{equation}
on $\mathbb{R}$, with specific behaviour at $t=s$ and $t \to \pm \infty$. In fact, we are looking for a pseudo-inverse of exponentially bounded solutions to \eqref{eq:inhomogeneous CMT}. To do this, consider for any $\eta \in (0, \min \{-a,b \})$ and $s \in \mathbb{R}$ the operator $\mathcal{K}_{s}^{\eta} : \BC_s^\eta(\mathbb{R},X^{\odot \star}) \to \BC_s^\eta(\mathbb{R},X)$ defined by
\begin{align}
\begin{split} \label{eq:Ketas}
    (\mathcal{K}_s^\eta f)(t) &\coloneqq j^{-1} \int_s^t T^{\odot \star}(t-\tau) P_0^{\odot \star} f(\tau) d\tau + j^{-1} \int_\infty^t T^{\odot \star}(t-\tau) P_+^{\odot \star} f(\tau) d\tau \\
    &+j^{-1} \int_{-\infty}^t T^{\odot \star}(t-\tau) P_-^{\odot \star} f(\tau) d\tau, \quad \forall f \in \BC_s^\eta(\mathbb{R},X^{\odot \star}).
\end{split}
\end{align}
The following result justifies that $\mathcal{K}_s^\eta$ is a well-defined bounded linear operator serving precisely as the desired pseudo-inverse. It is crucial to note that the operator norm $\|\mathcal{K}_{s}^{\eta}\| \leq \Omega_\eta$ is uniformly bounded above independent of $s$, a property that will play a key role in the proof of \Cref{thm:LipschitzCM}.

\begin{proposition}
Let $\eta \in (0,\min\{-a,b\})$ and $s \in \mathbb{R}$ be given. The following properties hold.
\begin{enumerate}
    \item $\mathcal{K}_{s}^{\eta}$ is a well-defined bounded linear operator. Moreover, the operator norm $\|\mathcal{K}_{s}^{\eta}\|$ is bounded above independent of $s$.
    \item $\mathcal{K}_{s}^{\eta}f$ is the unique solution of \eqref{eq:inhomogeneous CMT} in $\BC^{\eta}_s(\mathbb{R},X)$ with vanishing $X_0$-component at time $s$.
\end{enumerate}
\end{proposition}
\begin{proof}
Let $\varepsilon \in (0,\eta)$ be given and note that for a given $f \in \BC_s^\eta(\mathbb{R},X^{\odot \star})$, the three integrals appearing in \eqref{eq:Ketas} define functions $I_i : \mathbb{R} \to X^{\odot \star}$ for $i \in \{-,0,+\}$. We have to show that these functions are well-defined, continuous and take values in $j(X)$. Similar to the proofs of \cite[Proposition 3]{Article1} and \cite[Proposition 30]{Janssens2020}, we see that the integrands are weak$^\star$ continuous and obtain from \Cref{hyp:1} that
\begin{equation*}
    \|I_{-}(t)\| \leq  \frac{ K_\varepsilon e^{\eta|t-s|}}{-a-\eta}\|f\|_{\eta,s} , \quad \|I_{0}(t)\| \leq \frac{K_\varepsilon e^{\eta|t-s|}}{\eta - \varepsilon} \|f\|_{\eta,s}, \quad \|I_{+}(t)\| \leq \frac{K_{\varepsilon} e^{\eta |t-s|}}{b-\eta}  \|f\|_{\eta,s}, \quad \forall t \in \mathbb{R},
\end{equation*}
which proves that all integrals are well-defined. An application of \Cref{hyp:1} and \Cref{hyp:2} tells us that 
\begin{equation} \label{eq:Tsunstarj}
    T^{\odot \star}(t-\tau) P_i ^{\odot \star} f(\tau) = j T(t-\tau)j^{-1} P_i^{\odot \star} f(\tau), \quad \forall \tau \geq t, \ i \in \{0,+\},
\end{equation}
which shows that $I_{i}$ takes values in $j(X)$ for $i \in \{0,+\}$. To prove that $I_{-}$ takes values in $j(X)$, we rely on \cite[Lemma III.2.3]{Diekmann1995} and \cite[Lemma 9]{Janssens2020}, see also \cite[Theorem 1]{Article1} for a $T$-periodic version. From these references, we also conclude that $I_{i}$ is continuous for each $i \in \{-,0,+\}$, and thus
\begin{equation*}
    \|\mathcal{K}_s^\eta\| \leq \|j^{-1}\| K_\varepsilon \bigg( \frac{1}{\eta - \varepsilon} + \frac{1}{b-\eta} + \frac{1}{-a-\eta} \bigg) \eqqcolon \Omega_\eta < \infty,
\end{equation*}
which proves the first claim.

Let us now prove the second claim. A straightforward computation shows that $u = \mathcal{K}_s^\eta f$ is a solution of \eqref{eq:inhomogeneous CMT}. The fact that $u$ has vanishing $X_0$-component at time $s$ follows from a direct computation of $P_0 u(s)$ in combination with the mutual orthogonality of the projections $P_i$. To prove that this solution is unique, let $v \in \BC_s^\eta(\mathbb{R},X)$ be another solution of \eqref{eq:inhomogeneous CMT} satisfying $P_0 v(s) = 0$. Then, $w = u-v \in \BC_s^\eta(\mathbb{R},X)$ and $w(t) = T(t-s)w(s)$. \Cref{lemma:X0} tells us that $w \in X_0$ and thus $w(s) = P_0 w(s) = 0$, meaning that $w(t) = T(t-s)w(s) = 0$ for all $t \in \mathbb{R}$, and so $u = v$.
\end{proof}

\subsection{Modification of the nonlinearity} \label{subsec:modification}
To apply in \Cref{thm:LipschitzCM} a contraction argument for proving the existence of a Lipschitz continuous center manifold, it is necessary that the nonlinear operator $R(t,\cdot) : X \to X^{\odot \star}$ for fixed $t \in \mathbb{R}$ is globally Lipschitz. However, since $R$ is only assumed to be $C^k$-smooth, this condition is generally not satisfied. Nevertheless, since we are primarily concerned with the local behaviour of solutions near the origin, we may modify $R(t,\cdot)$ outside a ball of radius $\delta > 0$ to ensure that the operator becomes globally Lipschitz. To this end, we introduce a $C^\infty$-smooth cut-off function $\xi : [0,\infty) \to \mathbb{R}$ by
\begin{equation*}
    \xi(s) \in
    \begin{cases}
    \{ 1 \}, \quad &0 \leq s \leq 1, \\
    [0,1], \quad &1 \leq s \leq 2,\\
    \{ 0 \}, \quad & s \geq 2,
    \end{cases}
\end{equation*}
and define for any $\delta > 0$ the \emph{$\delta$-modification} of $R$ as the operator $R_{\delta} \in C^k(\mathbb{R} \times X, X^{\odot \star})$ with action
\begin{equation*}
   R_{\delta}(t,u) := R(t,u) \xi \bigg( \frac{\|P_0u\|}{\delta} \bigg) \xi \bigg( \frac{\|(P_{-} + P_{+})u\|}{\delta} \bigg), \quad \forall (t,u) \in \mathbb{R} \times X. 
\end{equation*}
Along the same lines of \cite[Proposition 32]{Janssens2020} and \cite[Proposition 4]{Article1}, one proves that for sufficiently small $\delta > 0$, the operator $R_{\delta}(t,\cdot)$ is globally Lipschitz continuous for any $t\in \mathbb{R}$ with Lipschitz constant $L_{\delta} \to 0$ as $\delta \downarrow 0$. Hence, one derives directly that $\|R_\delta(t,u)\| \leq 4 \delta L_\delta$ for all $(t,u) \in \mathbb{R} \times X$, see \cite[Corollary 3.1]{LentjesCMODE} for a proof. We associate with $R_\delta$ the \emph{substitution operator} $\tilde{R}_{\delta} : \BC_s^{\eta}(\mathbb{R},X) \to \BC_s^{\eta}(\mathbb{R},X^{\odot \star})$ by
\begin{equation*}
    \tilde{R}_{\delta}(u) := R_{\delta}(\cdot,u(\cdot)), \quad \forall u \in \BC_s^{\eta}(\mathbb{R},X),
\end{equation*}
and we note that $\tilde{R}_{\delta}$ inherits all Lipschitz properties from $R_{\delta}$, see \cite[Corollary 1]{Article1}. Let us now also introduce for any $\eta \in (0,\min \{-a,b \})$ and $s \in \mathbb{R}$ the bounded linear operator $T_s^\eta : X_0 \to \BC_s^\eta(\mathbb{R},X)$ by 
\begin{equation*}
    (T_s^\eta \varphi)(t) \coloneqq T(t-s)\varphi, \quad \forall \varphi \in X_0,
\end{equation*}
which is well-defined as $T$ extends to a $\mathcal{C}_0$-group on $X_{0}$ due to \Cref{hyp:1}. A straightforward estimate using \Cref{hyp:1} shows that the operator norm $\|T_{s}^{\eta}\| \leq K_\eta$ is uniformly bounded above independent of $s$, a property that will play a key role in the proof of \Cref{thm:LipschitzCM}.

\subsection{Existence of a Lipschitz center manifold}
Motivated by \Cref{lemma:X0}, we will define a parametrized fixed point operator such that its fixed points correspond to exponentially bounded solutions on $\mathbb{R}$ of the modified equation
\begin{equation} \label{eq:modifiedintegraleq}
    u(t) = T(t-s)u(s) + j^{-1} \int_s^t T^{\odot \star}(t-\tau) R_\delta(\tau,u(\tau)) d\tau, \quad- \infty < s \leq t < \infty,
\end{equation}
for some small $\delta > 0$. For a given $\eta \in (0,\min \{-a,b \})$ and $s \in \mathbb{R}$, let us introduce the fixed point operator $\mathcal{G}_s^\eta : \BC_s^\eta(\mathbb{R},X) \times X_0 \to \BC_s^\eta(\mathbb{R},X^{\odot \star})$ by
\begin{equation*}
   \mathcal{G}_s^\eta(u,\varphi) := T_s^\eta\varphi + \mathcal{K}_s^\eta(\tilde{R}_{\delta}(u)), \quad \forall (u,\varphi) \in \BC_s^\eta(\mathbb{R},X) \times X_0,
\end{equation*}
where the second argument must be treated as a parameter.

\begin{theorem} \label{thm:LipschitzCM}
Let $\eta \in (0,\min \{-a,b \})$ and $s\in \mathbb{R}$ be given. If $\delta > 0$ is sufficiently small, then the following two statements hold.
\begin{enumerate}
    \item For every $\varphi \in X_0$ the equation $ u = \mathcal{G}_s^\eta(u,\varphi)$ has a unique solution $u = u_s^\star(\varphi)$.
    \item The map $u_s^\star : X_0 \to \BC_s^\eta(\mathbb{R},X)$ is globally Lipschitz and satisfies $u_s^\star(0) = 0$.
\end{enumerate}
\end{theorem}
\begin{proof}
Take $(u,\varphi),(v,\psi) \in \BC_s^\eta (\mathbb{R},X) \times X_0$ and note that
\begin{equation*}
    \|\mathcal{G}_s^\eta(u,\varphi) - \mathcal{G}_s^\eta(v,\psi)\|_{\eta,s} \leq K_\eta\|\varphi- \psi\| + \Omega_\eta L_{\delta} \|u-v\|_{\eta,s}.
\end{equation*}
By the results from \Cref{subsec:modification}, we know that there exists a $\delta > 0$ small enough such that $\Omega_\eta L_{\delta} \leq \frac{1}{2}$. 

To prove the first assertion, note that
\begin{equation*}
    \|\mathcal{G}_s^\eta(u,\varphi) - \mathcal{G}_s^\eta(v,\varphi)\|_{\eta,s} \leq \frac{1}{2} \|u-v\|_{\eta,s},
\end{equation*}
and so $\mathcal{G}_s^\eta(\cdot,\varphi)$ is a contraction on the Banach space $\BC_s^\eta(\mathbb{R},X)$ equipped with the $\|\cdot\|_{\eta,s}$-norm. The contraction mapping principle states that $\mathcal{G}_s^\eta(\cdot,\varphi)$ has a unique fixed point $u_s^\star(\varphi)$.

To prove the second assertion, consider the unique fixed points $u_s^\star(\varphi)$ and $u_s^\star(\psi)$ of $\mathcal{G}_s^\eta(u,\varphi)$ and $\mathcal{G}_s^\eta(u,\psi)$, respectively. Then,
\begin{equation*}
    \|u_s^\star(\varphi) - u_s^\star(\psi)\|_{\eta,s} 
    \leq K_\eta \|\varphi - \psi\| + \frac{1}{2}\|u_s^\star(\varphi) - u_s^\star(\psi)\|_{\eta,s},
\end{equation*}
which shows that $\|u_s^\star(\varphi) - u_s^\star(\psi)\|_{\eta,s} \leq 2 K_\eta \|\varphi - \psi\|$, and so $u^\star_s$ is globally Lipschitz. Since $u_s^\star(0) = \mathcal{G}_s^\eta(u_s^\star(0),0) = 0$, the second assertion follows.
\end{proof}
Let us now define the map $\mathcal{C} : \mathbb{R} \times X_0 \to X$ by $\mathcal{C}(s,\varphi) \coloneqq u_s^\star(\varphi)(s)$ and introduce the sets
\begin{equation*}
    \mathcal{W}^c \coloneqq \{(s,\varphi) \in \mathbb{R} \times X : \varphi \in \mathcal{W}^c(s) \}, \quad \mathcal{W}^c(s) \coloneqq \{ \mathcal{C}(s,\varphi) \in X : \varphi \in X_0\}, \quad \forall s \in \mathbb{R}.
\end{equation*}
We call $\mathcal{W}^c$ the \emph{global center manifold} of \eqref{eq:modifiedintegraleq} whose \emph{$s$-fibers} are given by $\mathcal{W}^c(s)$. From \Cref{thm:LipschitzCM}, we observe directly that $\mathcal{C}(s,\cdot)$ is globally Lipschitz, where the Lipschitz constant might depend on $s$, meaning that $\mathcal{C}$ is only fiberwise Lipschitz. However, along the same lines of \cite[Corollary 3]{Article1} one can prove that this Lipschitz constant can be chosen independently of $s$, meaning that $\mathcal{W}^c$ is in fact a global Lipschitz center manifold of \eqref{eq:modifiedintegraleq}.

Let $B_\delta(X)$ be the open ball centered around the origin in $X$ of radius $\delta > 0$. Then, \Cref{subsec:modification} shows that $R_\delta = R$ on $\mathbb{R} \times B_\delta(X)$, meaning that the modified integral equation \eqref{eq:modifiedintegraleq} is equivalent to the original integral equation \eqref{eq:inhomogenous} on $B_\delta(X)$. Then, we call
\begin{equation*}
    \mathcal{W}_{\loc}^c := \{(s,\varphi) \in \mathbb{R} \times X : \varphi \in \mathcal{W}^c(s) \cap B_\delta(X) \}
\end{equation*}
the \emph{local center manifold} of \eqref{eq:inhomogenous} for some sufficiently small $\delta > 0$. In the framework of nonautonomous dynamical systems, the sets $\mathcal{W}^c$ and $\mathcal{W}_{\loc}^c$ are also called \emph{integral manifolds}.

\subsection{Properties of the center manifold} \label{sec:properties}

Before we prove the main theorem of this section (\Cref{thm:LCMT}), let us first show that $\mathcal{W}^c(s)$ is indeed a nonlinear generalization of the center subspace $X_0$, which was characterized in \Cref{lemma:X0}. The proof is a straightforward combination of \cite[Proposition 36]{Janssens2020} and \cite[Proposition 5]{Article1} (or \cite[Proposition 6]{LentjesCMODE}), and therefore omitted.
\begin{proposition} \label{prop:Wcs}
Let $\eta \in (0,\min\{-a,b\})$ and $s \in \mathbb{R}$ be given. Then
\begin{equation*}
    \mathcal{W}^c(s) = \{ \varphi \in X : \mbox{there exists a solution of \eqref{eq:modifiedintegraleq} on $\mathbb{R}$ through $\varphi$ at time $s$ belonging to } \BC_{s}^{\eta}(\mathbb{R},X) \}.
\end{equation*}
\end{proposition}
The local invariance property of $\mathcal{W}_{\loc}^c$ is proven in the following result. Again, since this proof of this result is just a straightforward combination of \cite[Corollary 38]{Janssens2020} (or \cite[Theorem IX.5.3]{Diekmann1995}) and \cite[Theorem 2]{Article1} (or \cite[Proposition 6]{LentjesCMODE}), it is simply omitted.

\begin{proposition} \label{prop:locinv}
The local center manifold $\mathcal{W}_{\loc}^c$ has the following properties.
\begin{enumerate}
    \item $\mathcal{W}_{\loc}^c$ is locally positively invariant: if $(s,\varphi) \in \mathcal{W}_{\loc}^c$ and $s < t_\varphi \leq \infty$ are such that $S(t,s,\varphi) \in B_\delta(X)$ for all $t \in [s,t_\varphi),$ then $(t,S(t,s,\varphi)) \in \mathcal{W}_{\loc}^c$.
    \item $\mathcal{W}_{\loc}^c$ contains every solution of \eqref{eq:inhomogenous} that exists on $\mathbb{R}$ and remains sufficiently small for all positive time and negative time: if $u : \mathbb{R} \to B_\delta(X)$ is a solution of \eqref{eq:inhomogenous}, then $(t,u(t)) \in \mathcal{W}_{\loc}^c$ for all $t\in \mathbb{R}$.
    \item If $(s,\varphi) \in \mathcal{W}_{\loc}^c$, then $S(t,s,\varphi) = u_t^\star(P_0S(t,s,\varphi))(t) = \mathcal{C}(t,P_0S(t,s,\varphi))$ for all $t \in [s,t_\varphi)$.
    \item $\mathbb{R} \times \{0\} \subseteq \mathcal{W}_{\loc}^c$ and $\mathcal{C}(t,0) = 0$ for all $t\in \mathbb{R}$.
\end{enumerate}
\end{proposition}

The next objective is to demonstrate that the mapping $\mathcal{C}$ inherits the same order of smoothness as the time-dependent nonlinear perturbation $R$. Establishing higher regularity of center manifolds is a nontrivial task. A classical approach to increase smoothness is through the theory of contractions on scales of Banach spaces \cite{Vanderbauwhede1987}. By employing this method, one can follow arguments similar to those in \cite{Diekmann1991CM,Diekmann1995,Article1} to show that the map $\mathcal{C}$ is of class $C^k$. We also refer to \cite{LentjesCMODE,Hupkes2006,Hupkes2008,Church2018,Church2021} for further applications of this technique to various types of evolution equations. Consequently, we conclude that $\mathcal{W}^c$ and $\mathcal{W}^c_{\loc}$ are $C^k$-smooth manifolds in $\mathbb{R} \times X$. 

The additional regularity of the center manifolds enables a precise analysis of their tangency properties. Along the same lines of \cite[Corollary 5]{Article1} or \cite[Proposition 8]{LentjesCMODE}, we derive that $\mathbb{R} \times X_0$ is the tangent bundle restricted to the zero section of $\mathcal{W}^c$ and $\mathcal{W}^c_{\loc}$: $D_2 \mathcal{C}(s,0)\varphi = \varphi$ for all $(s,\varphi) \in \mathbb{R} \times X_0$.

Since $R$ is $T$-periodic in the first component, one expects naturally that (the second component of) $\mathcal{W}^c$ and $\mathcal{W}^c_{\loc}$ are itself $T$-periodic manifolds in $\mathbb{R} \times X$, meaning that there exists a $\delta > 0$ such that $\mathcal{C}(s+T,\varphi) = \mathcal{C}(s,\varphi)$ for all $(s,\varphi) \in \mathbb{R} \times B_\delta(X_0)$. This can be proven along the same lines of \cite[Theorem 6]{Article1} or \cite[Proposition 9]{LentjesCMODE}. Combining all the results from above, we arrive at the following result.

\begin{theorem}[{Local center manifold}] \label{thm:LCMT}
Let $T_0$ be a $\mathcal{C}_0$-semigroup on a real $\odot$-reflexive Banach space $X$ and let $T$ be the $\mathcal{C}_0$-semigroup defined by \eqref{eq:LAIE} satisfying \emph{\Cref{hyp:1}} and \emph{\Cref{hyp:2}}, where $B : X \to X^{\odot \star}$ is a bounded linear perturbation. Suppose in addition that $X_0$ is $n_0$-dimensional and that the nonlinear perturbation $R : \mathbb{R} \times X \to X^{\odot \star}$ satisfying \eqref{eq:Rassumptions} is $C^k$-smooth for some $k \geq 1$ and $T$-periodic in the first component. Then there exists a $\delta > 0$ and a $C^k$-smooth map $\mathcal{C} : \mathbb{R} \times X_0 \to X$ such that the manifold $\mathcal{W}_{\loc}^c := \{(s,\mathcal{C}(s,\varphi)) \in \mathbb{R} \times X : \mathcal{C}(s,\varphi) \in B_\delta(X) \text{ and } \varphi \in X_0 \}$ is $T$-periodic, $C^k$-smooth, $(n_0+1)$-dimensional, locally positively invariant for the time-dependent semiflow $S$ generated by \eqref{eq:inhomogenous} whose tangent bundle restricted to the zero section is $\mathbb{R} \times X_0$.
\end{theorem}

\begin{remark} \label{remark:LCMTextension}
The assumptions in \Cref{thm:LCMT} could be slightly relaxed to prove a similar result. However, to maintain simplicity and avoid introducing additional technical complications above, we chose not to pursue this. For instance, it would suffice to assume that the nonlinear perturbation $R$ is only $C^k$-smooth in an open neighbourhood $\mathbb{R} \times U$ of $\mathbb{R} \times \{0\}$. Moreover, if the assumption of $T$-periodicity in the first component of $R$ is dropped, but one instead assumes that $\sup_{\varphi \in B_\varepsilon(X)}\|R(t,\varphi)\| = \mathcal{O}(\varepsilon^2)$ and $\sup_{\varphi \in B_\varepsilon(X)}\|D_2 R(t,\varphi)\| = \mathcal{O}(\varepsilon)$ for all $t \in \mathbb{R}$ and any sufficiently small $\varepsilon > 0$, then the same result remains valid, but the $T$-periodicity of $\mathcal{W}^c$ and $\mathcal{W}^c_{\loc}$ is, in general, no longer preserved. \hfill $\lozenge$
\end{remark}

Center manifolds exhibit dynamical properties that are crucial for applications. A key feature for modelling and simulation (\Cref{sec:Examples}) is the \emph{local attractivity} of the center manifold: if $\sigma_{+} = \emptyset$, then any solution that remains sufficiently close to the center manifold forward in time is exponentially attracted to a specific solution on this manifold. To be precise, if $u$ is a solution of \eqref{eq:inhomogenous} such that $(t,u(t))$ remains in a sufficiently small neighbourhood of $\mathcal{W}_{\loc}^c$, then there exists a solution $v$ of \eqref{eq:inhomogenous} and constants $K, c > 0$ such that $(t,v(t)) \in \mathcal{W}_{\loc}^c$ and $\|u(t) - v(t)\| \leq Ke^{-c(t-s)}$ for all $t \geq s$. This result can be established by following the same contraction argument as used in the autonomous finite-dimensional ODE framework \cite[Theorem 2]{Carr1981}, and is therefore omitted. A crucial point in this non-autonomous setting is that the Lipschitz constants, arising in the center manifold construction, are $s$-independent due to the $T$-periodicity of the nonlinearity $R$ in the first component. For a comprehensive argument in the context of impulsive delay differential equations, see the proof of \cite[Theorem I.5.5.1]{Church2021}.

To complement the attractivity property, the explicit derivation of the reduced dynamics is based on the \emph{approximation theorem} of the center manifold, which is employed in the proof of \Cref{thm:periodicnormalform}. In the present nonlinearly periodically forced setting, the local center manifold $\mathcal{W}_{\loc}^c$ is described by the map $\mathcal{C}$, which is $T$-periodic in the first component on $\mathbb{R} \times B_\delta(X_0)$ for some sufficiently small $\delta > 0$. If one constructs a sufficiently smooth candidate function $C : \mathbb{R} \times B_\delta(X_0) \to X$ of $\mathcal{C}$, which is also $T$-periodic in the first component and satisfies the governing center manifold invariance equation up to a residual error of $\mathcal{O}(\|\varphi\|^p)$, then the exact manifold parametrization satisfies $\|\mathcal{C}(s,\varphi) - C(s,\varphi)\| = \mathcal{O}(\|\varphi\|^p)$ uniformly in $s$. The formal proof of this asymptotic accuracy can be obtained using the contraction argument as used in \cite[Theorem 3]{Carr1981} for autonomous finite-dimensional ODEs. Because the necessary bounds rely fundamentally on the contraction mapping principle discussed above and the bounded $s$-independent Lipschitz constants arising from the $T$-periodicity, the detailed argument is similarly omitted.

\subsection{The special case of classical DDEs} \label{sec:DDEs}
Our next aim is to show that the nonlinearly periodically forced system \eqref{eq:DDE} admits a center manifold around the equilibrium $\overline{x}$. As already announced in \Cref{sec:introduction}, we will first prove by an application of \Cref{thm:LCMT} that the translated equation \eqref{eq:introDDE3} admits a center manifold around the origin. The starting point to study classical DDEs on $X = C([-h,0],\mathbb{R}^n)$ is the shift semigroup $T_0$ defined by
\begin{equation*}
    (T_0(t)\varphi)(\theta) := 
    \begin{cases}
    \varphi(t+\theta), \quad &-h \leq t+ \theta \leq 0,\\
    \varphi(0), \quad &t+\theta \geq 0,
    \end{cases}
    \quad \forall \varphi \in X, \ t \geq 0, \ \theta \in [-h,0],
\end{equation*}
which generates the solution corresponding to the trivial DDE $\dot{x}(t) = 0$ for $t \geq 0$ with initial condition $x_0 = \varphi$, through the relation $x_t = T_0(t)\varphi$ for all $t \geq 0$. A representation theorem by Riesz \cite{Riesz1914} enables us to identify $X^\star = C([-h,0],\mathbb{R}^n)^\star$ with the Banach space $\NBV([0,h],\mathbb{R}^{n\star})$ consisting of functions $\zeta : [0,h] \to \mathbb{R}^{n \star}$ that are normalized by $\zeta(0) = 0$, are continuous from the right on $(0,h)$ and have bounded variation. The remaining spaces can be characterized by
\begin{equation} \label{eq:spacesDDEs}
    X^{\odot} \cong \mathbb{R}^{n \star} \times L^1([0,h],\mathbb{R}^{n\star}), \quad X^{\odot \star} \cong \mathbb{R}^n \times L^\infty([-h,0],\mathbb{R}^n), \quad X^{\odot \odot} \cong \mathbb{R}^n \times C([-h,0],\mathbb{R}^n),
\end{equation}
where $\cong$ denotes an explicit isometric isomorphism. The natural dual pairing $\langle \cdot,\cdot \rangle$ between $\varphi^\odot = (c,g) \in X^{\odot}$ and $\varphi \in X$, and $\varphi^{\odot \star} = (a,\psi) \in X^{\odot \star}$ and $\varphi^\odot$, is given by
\begin{equation} \label{eq:naturalpairings}
    \langle \varphi^\odot,\varphi\rangle = c \varphi(0) + \int_0^h g(\theta) \varphi(-\theta) d\theta, \qquad \langle \varphi^{\odot \star},\varphi^\odot \rangle = ca + \int_0^h g(\theta) \psi(-\theta) d\theta.
\end{equation}
The canonical embedding $j$ from \eqref{eq:j} has action $j\varphi = (\varphi(0),\varphi)$ for all $\varphi \in X$, mapping $X$ onto $X^{\odot \odot}$, meaning that $X$ is $\odot$-reflexive with respect to the shift semigroup $T_0$. We refer to \cite[Section 2]{Diekmann1995} for the proofs of the above statements.

In this setting, the bounded linear perturbation $B : X \to X^{\odot \star}$ takes the form $B\varphi = [L \varphi]r^{\odot \star} $ while the nonlinear $C^k$-smooth perturbation $R : \mathbb{R} \times X \to X^{\odot \star}$, which is $T$-periodic in the first component, is given by $R(t,\varphi) = G(t,\varphi)r^{\odot \star}$ for all $(t,\varphi) \in \mathbb{R} \times X$. Here, we used the conventional shorthand notation $wr^{\odot \star} \coloneqq \sum_{i=1}^n w_i r_i^{\odot \star}$ for all $w = (w_1,\dots,w_n) \in \mathbb{R}^n$, where $r_i^{\odot \star} \coloneqq (e_i,0) \in X^{\odot \star}$ and $e_i$ denotes the $i$th basis vector of $\mathbb{R}^n$ for all $i = 1,\dots,n$. The $\mathcal{C}_0$-semigroup $T$ generates solutions of the linear DDE \eqref{eq:introDDE3} with $G = 0$ via the relation $y_t = T(t-s)\varphi$ for $t \geq s$. Similarly, combining the proofs of \cite[Proposition VII.6]{Diekmann1995} and \cite[Theorem 7]{Article1}, we find that the time-dependent semiflow $S$ from \eqref{eq:Sprocess} generates solutions of \eqref{eq:introDDE3} via the relation $y_t = S(t,s,\varphi)$ for $t \in [s,t_\varphi)$.

It remains to verify \Cref{hyp:1} and \Cref{hyp:2} in order to apply \Cref{thm:LCMT}, for which we rely on \Cref{thm:spectralgab}. In the setting of classical DDEs, it is known that $T$ is eventually compact ($T_0$ is eventually compact and $B$ is bounded) and thus eventually norm continuous. Since the associated generator $A$ given by
\begin{equation*}
    \mathcal{D}(A) = \{ \varphi \in X : \varphi' \in X \mbox{ and } \varphi'(0) = L\varphi\}, \quad A\varphi = \varphi',
\end{equation*}
has compact resolvent, the spectrum $\sigma(A)$ consists solely of isolated eigenvalues of finite type. To further characterize $\sigma(A)$, we begin by applying a vector-valued version of the Riesz representation theorem to the linear operator $L$ yielding
\begin{equation*}
    L\varphi = \int_0^h d\zeta(\theta)\, \varphi(-\theta), \quad \forall \varphi \in X,
\end{equation*}
where the kernel $\zeta \in \NBV([0,h], \mathbb{R}^{n \times n})$, and the integral is understood in the Riemann-Stieltjes sense. With this representation, the spectrum of $A$ can be described as
\begin{equation} \label{eq:spectraDelta}
    \sigma(A) = \{ z \in \mathbb{C} : \det \Delta(z) = 0 \},
\end{equation}
where the \emph{characteristic matrix} $\Delta(z) \in \mathbb{C}^{n \times n}$ is defined by
\begin{equation} \label{eq:Delta}
    \Delta(z) \coloneqq zI - \int_0^h e^{-z\theta} \, d\zeta(\theta).
\end{equation}
Further (spectral) properties of $\Delta(z)$ and its connection with $A$ can be found in \cite{Kaashoek1992,Diekmann1995,Hale1993}. An immediate consequence of \eqref{eq:spectraDelta} is that $\sigma(A)$ contains only finitely many elements in any right half-plane of the form $\{ z \in \mathbb{C} : \Re(z) > \gamma \}$ for fixed $\gamma \in \mathbb{R}$, thereby ensuring that all conditions of \Cref{thm:spectralgab} are satisfied. Hence, we conclude that \eqref{eq:introDDE3} admits a center manifold $\mathcal{W}_{\loc}^c$ around $\mathbb{R} \times \{0\}$. Since \eqref{eq:DDE} was just a translation of \eqref{eq:introDDE3} by $x = \overline{x} + y$, it follows immediately that
\begin{equation} \label{eq:Wclocxbar}
    \mathcal{W}_{\loc}^c(\overline{x}) \coloneqq \{ (s,\overline{x} + \mathcal{C}(s,\varphi)) \in \mathbb{R} \times X : \mathcal{C}(s,\varphi) \in B_\delta(X) \mbox{ and } \varphi \in X_0  \}
\end{equation}
is a $T$-periodic $C^k$-smooth $(n_0+1)$-dimensional locally positively invariant manifold in $\mathbb{R} \times X$, defined in the vicinity of $\mathbb{R} \times \{\overline{x}\}$ for a sufficiently small $ \delta > 0$, whose tangent bundle restricted to the zero section is $\mathbb{R} \times X_0$. We call $\mathcal{W}_{\loc}^c(\overline{x})$ a \emph{local center manifold around $\overline{x}$}.

\begin{remark} \label{remark:CMODE}
If we take the maximal delay $h \downarrow0$, then $X$ can be identified with $\mathbb{R}^n$ and the sun-star framework becomes trivial. In this case, \eqref{eq:DDE} simplifies to the nonlinearly periodically forced ODE $\dot{x}(t) = F(t, x(t))$ for $t \geq s$. The center manifold $\mathcal{W}_{\loc}^c(\overline{x})$ for this ODE was already constructed in \cite[Section 2.3.2]{Haragus2011}. Furthermore, explicit (counter)examples of (non-)unique, (non-)$C^\infty$-smooth and (non-)analytic periodic center manifolds for nonlinearly periodically forced ODEs are given in \cite[Section 5]{LentjesCMODE}. Thus, \Cref{thm:LCMT} can be seen as a natural extension of the cited center manifold theorems. \hfill $\lozenge$
\end{remark}

\section{Periodically forced normal form theorem} \label{sec:NFtheorem}

Let us reconsider the setting described in \Cref{sec:DDEs}, where we are in the situation of an $n_0$-dimensional center subspace $X_0$ with $(n_0 + 1)$-dimensional center manifold $\mathcal{W}_{\loc}^c(\overline{x})$. Recall from \Cref{sec:introduction} that our next aim is to study the dynamics of \eqref{eq:DDE} on $\mathcal{W}_{\loc}^c(\overline{x})$ by the means of periodically forced normal forms. Before proving the main result of this section (\Cref{thm:periodicnormalform}), some preliminary steps are necessary.

Note that $X_0 = \oplus_{\lambda \in \sigma_0} E_\lambda$, where $E_\lambda \coloneqq \mathcal{N}((\lambda I - A)^{k_\lambda})$ denotes the $m_\lambda$-dimensional generalized eigenspace of $A$ corresponding to the eigenvalue $\lambda$ and $k_\lambda$ denotes the pole of the resolvent $z \mapsto (zI-A)^{-1}$ of $A$ at $\lambda$. Let $\{\varphi_0, \dots, \varphi_{m_\lambda - 1}\}$ denote the associated Jordan basis consisting of $p$ Jordan chains of maximal rank with $p$ being the dimension of the eigenspace $\mathcal{N}(\lambda I - A)$. Let $J_\lambda \coloneqq \text{diag}(J_{\lambda,1},\dots,J_{\lambda,p})$ be the associated Jordan matrix for $\lambda$, where $J_{\lambda,j}$ denotes the $j$th Jordan block corresponding to the $j$th Jordan chain. We define $M_0 \coloneqq \text{diag}(J_{\lambda_{1}}, \dots, J_{\lambda_{l}})$ to be the Jordan normal form, where $l$ is the number of distinct eigenvalues of $A$ on the imaginary axis. After relabelling, we denote by $\{\varphi_1, \dots, \varphi_{n_0}\}$ the Jordan basis of $X_0$, which allows us to define the coordinate map $Q_0 : \mathbb{R}^{n_0} \to X_0$ by
\begin{equation*}
    Q_0 \xi \coloneqq \sum_{i=1}^{n_0} \xi_i \varphi_i, \quad \forall \xi = (\xi_1,\dots,\xi_{n_0}) \in \mathbb{R}^{n_0}.
\end{equation*}
As indicated in \Cref{sec:introduction}, our first objective is to express any solution $x_t$ of \eqref{eq:DDE} on $\mathcal{W}_{\loc}^c(\overline{x})$ using the natural parametrization \eqref{eq:parametrization} in terms of a normal form coordinate $\xi \in \mathbb{R}^{n_0}$. In this formulation, the only unknown is now the (nonlinear) operator $H : \mathbb{R} \times \mathbb{R}^{n_0} \to X$. To establish the existence and smoothness of this operator, we rely on the invariance property of $\mathcal{W}_{\loc}^c(\overline{x})$. Specifically, let $s \in \mathbb{R}$ be a starting time with $(s,x_s) = (s,\varphi) \in \mathcal{W}_{\loc}^c(\overline{x})$, and let the interval $I \subseteq I_\varphi$, containing $s$, denote the domain of definition of the solution (\Cref{prop:locinv}). Then, a slight modification of \cite[Theorem 3.6]{Clement1989} tells us that $u(t) \coloneqq x_t$ satisfies the abstract ODE
\begin{equation} \label{eq:DDEODE}
    j \dot{u}(t) = A_0^{\odot \star} ju(t) + W(t,x_t), \quad \forall t \in I,
\end{equation}
where $W \in C^{k}(\mathbb{R} \times X, X^{\odot \star})$, which is $T$-periodic in the first component, is defined by $W(t,\varphi) \coloneqq F(t,\varphi) r^{\odot \star}$ for all $(t,\varphi) \in \mathbb{R} \times X$, and the weak$^\star$ generator $A_0^{\odot \star}$ of $T_0^{\odot \star}$ takes the form
\begin{equation*}
    \mathcal{D}(A_0^{\odot \star}) = \{(\alpha,\varphi) \in X^{\odot \star} : \varphi \in \text{Lip}([-h,0],\mathbb{R}^n) \mbox{ and } \varphi(0) = \alpha \}, \quad A_0^{\odot \star} (\alpha,\varphi) = (0,\varphi').
\end{equation*}
In the proof of the theorem below, it proves useful to have an explicit representation of $A^{\odot \star}$. By recalling \eqref{eq:Asunstar}, we obtain
\begin{equation} \label{eq:AsunstarDDE}
    \mathcal{D}(A^{\odot \star}) = \mathcal{D}(A_0^{\odot \star}), \quad A^{\odot \star} (\alpha,\varphi) = (L\varphi,\varphi').
\end{equation}
The following result is the main theorem of this section. It shows that the natural parametrization \eqref{eq:parametrization} indeed holds, where the normal form coordinate $\xi$ satisfies the periodically forced ODE \eqref{eq:Normalformxi} in $\mathbb{R}^{n_0}$. Furthermore, the nonlinear part $P$ in \eqref{eq:Normalformxi} can be characterized polynomially up to order $k$ via the relation \eqref{eq:restrictionP}.

\begin{theorem} \label{thm:periodicnormalform}
There exist $C^k$-smooth maps $H : \mathbb{R} \times \mathbb{R}^{n_0} \to X$ and $P : \mathbb{R} \times \mathbb{R}^{n_0} \to \mathbb{R}^{n_0}$ such that any solution $x_t$ of \eqref{eq:DDE} on $\mathcal{W}_{\loc}^c(\overline{x})$ may be represented as
\begin{equation*}
    x_t = \overline{x} + Q_0\xi + H(t,\xi), \quad \forall t \in I,
\end{equation*}
where the dynamics of \eqref{eq:DDE} on $\mathcal{W}_{\loc}^c(\overline{x})$ is defined by the normal form
\begin{equation} \label{eq:Normalformxi}
    \dot{\xi} = M_0 \xi + P(t,\xi) + \mathcal{O}(|\xi|^{k+1}).
\end{equation}
Here the functions $H$ and $P$ are $T$-periodic in $t$ and at least quadratic in $\xi$, while the $\mathcal{O}$-terms are also $T$-periodic. Moreover, $P$ is a polynomial in $\xi$ of degree less than or equal to $k$ satisfying
\begin{equation} \label{eq:restrictionP}
    e^{t M_0^\star } P(t,e^{-tM_0^\star}\xi) = P(0,\xi), \quad \forall (t,\xi) \in \mathbb{R} \times \mathbb{R}^{n_0}.
\end{equation}
\end{theorem}
\begin{proof}
The proof of this theorem is divided into several steps.

\textbf{Step 1: Existence.} Consider a point $(t_0,\Psi_0) \in \mathcal{W}_{\loc}^c(\overline{x})$ together with an associated fixed pair $(t_0,\xi_0) \in \mathbb{R} \times \mathbb{R}^{n_0}$ such that \eqref{eq:decomposition X hyp} allows us to write
\begin{equation} \label{eq:Psi0}
    \Psi_0 = \overline{x} + Q_0\xi_0 + \hat{H}(t_0,\xi_0),
\end{equation}
with $Q_0 \xi_0 \in X_0$ and $\hat{H}(t_0,\xi_0) \in X_{-} \oplus X_{+}$. Assume for a moment that $\hat{H}$ is $C^1$-smooth, which will be proven later, independent on the upcoming construction. Consider then the continuously differentiable operator $\mathcal{K} : \mathbb{R} \times \mathbb{R}^{n_0} \times X \to X$ defined by
\begin{equation*}
    \mathcal{K}((t,\xi),\Psi) \coloneqq \Psi - (\overline{x} + Q_0\xi + \hat{H}(t,\xi)),
\end{equation*}
and note from \eqref{eq:Psi0} that $\mathcal{K}((t_0,\xi_0),\Psi_0) = 0$. Since $D_\Psi\mathcal{K}((t_0,\xi_0),\Psi_0) = I_X$, the implicit theorem in Banach spaces guarantees that there exists a continuously differentiable function $\Psi : \mathbb{R}\times \mathbb{R}^{n_0} \to X$, defined in a sufficiently small neighbourhood of $\mathbb{R} \times \{0\}$ such that $\mathcal{K}((t,\xi),\Psi(t,\xi)) = 0$, or equivalently,
\begin{equation} \label{eq:Psi(t,xi)}
    \Psi(t,\xi) = \overline{x} + Q_0\xi + \hat{H}(t,\xi),
\end{equation}
which proves that $\Psi$ is a parametrization of (the second component of) $\mathcal{W}_{\loc}^c(\overline{x})$. Similar to the proof of \cite[Theorem 1]{Iooss1988} and \cite[Theorem III.7]{Iooss1999}, we allow $\hat{H}(t,\cdot)$ taking values in $X_0$. The map $\hat{H}$ with this additional property will be denoted by $H$. Hence, we incorporate an eventual nonlinear change of coordinates on $\xi$ in \eqref{eq:Psi(t,xi)} in such a way that \eqref{eq:DDE} written on $\mathcal{W}_{\loc}^c(\overline{x})$ is already in normal form. The remaining claims will be proven in the following steps.

\textbf{Step 2: Taylor expansion.} Let us write \eqref{eq:DDE} in the form \eqref{eq:DDEODE} and notice that
\begin{equation} \label{eq:TaylorexpansionW}
    W(t,\overline{x} + \varphi) = W(t,\overline{x}) + B\varphi + \sum_{q=2}^k W_q(t,\varphi^{(q)}) + \mathcal{O}(\|\varphi\|_{\infty}^{k+1}),
\end{equation}
where $B\varphi = [L\varphi]r^{\odot \star}$ with $L = D_2F(t,\overline{x})$ is the bounded linear perturbation, and the nonlinear terms are given by $W_q(t,\varphi^{(q)}) = F_q(t,\varphi^{(q)})r^{\odot \star}$ with $F_q(t,\varphi^{(q)}) \coloneqq \frac{1}{q!} D_2^qF(t,\overline{x})(\varphi^{(q)})$. Here, $D_2^qF(t,\overline{x}) : X^q \to \mathbb{R}^n$ denotes the partial $q$th order Fr\'echet derivative with respect to the second component evaluated at $(t,\overline{x})$, and $\varphi^{(q)} \coloneqq (\varphi,\dots,\varphi) \in X^q \coloneqq X \times \dots \times X$ for $q \in \{2,\dots,k\}$. We also expand the maps $H$ and $P$ accordingly
\begin{equation*}
    H(t,\xi) = \sum_{q=2}^{k} H_q(t,\xi^{(q)}) + \mathcal{O}(|\xi|^{k+1}), \quad P(t,\xi) = \sum_{q=2}^k P_q(t,\xi^{(q)}),
\end{equation*}
with coefficients $H_q(t,\xi^{(q)}) \in X$ and $P_q(t,\xi^{(q)}) \in \mathbb{R}^{n_0}$, where $\xi^{(q)} \coloneqq (\xi,\dots,\xi) \in [\mathbb{R}^{n_0}]^q$. To obtain a representation of $H_q(t,\xi^{(q)})$, we use the invariance property of $\mathcal{W}_{\loc}^c(\overline{x})$ from \Cref{prop:locinv} and therefore start comparing the expansions of
\begin{equation*}
    \frac{d}{dt} j(\overline{x} + Q_0\xi + H(t,\xi)) = j \bigg( Q_0 \dot{\xi} + \frac{\partial}{\partial t} H(t,\xi) +  D_\xi H(t,\xi) \dot{\xi} \bigg)
\end{equation*}
and
\begin{equation*}
    A_{0}^{\odot \star} j(\overline{x} + Q_0\xi + H(t,\xi)) + W(t,\overline{x} + Q_0\xi + H(t,\xi))
\end{equation*}
by substituting $\dot{\xi}$ by \eqref{eq:Normalformxi} and recalling \eqref{eq:TaylorexpansionW}. Using the expansions of $W,H$ and $P$, we get
\begin{align*}
    &j\bigg( \sum_{q=2}^k \frac{\partial}{\partial t} H_q(t,\xi) + \bigg(Q_0 + \sum_{q=2}^k D_\xi H_q(t,\xi^{(q)}) \bigg) \bigg( M_0 \xi + \sum_{q=2}^k P_q(t,\xi^{(q)}) \bigg) \bigg) + \mathcal{O}(|\xi|^{k+1}) \\
    &= A_0^{\odot \star} j(\overline{x}) + W(t,\overline{x}) + A^{\odot \star}j \bigg( Q_0 \xi + \sum_{q=2}^k H_q(t,\xi^{(q)}) \bigg) + \sum_{q=2}^k W_q\bigg( t, \bigg[Q_0 \xi + \sum_{p=2}^k H_q(t,\xi^{(p)}) \bigg]^{(q)} \bigg).
\end{align*}
\textbf{Step 3: Collecting terms.} Let us now compare the $\xi^{(q)}$-terms on both sides of this equation for $q \in \{0,\dots,k\}$. The constant and linear terms give us
\begin{equation} \label{eq:xixi2 terms}
    W(t,\overline{x}) = 0, \quad A^{\odot \star} j(Q_0 \xi) = j(M_0Q_0 \xi).
\end{equation}
The first identity confirms that $\overline{x}$ is an equilibrium of \eqref{eq:DDEODE} and thus of \eqref{eq:DDE}, while the second identity is also confirmed since $A Q_0 \xi = M_0 Q_0 \xi$ and $Q_0 \xi \in \mathcal{D}(A)$ as $X_0 \subseteq \mathcal{D}(A)$ by the Jordan chain structure of the generalized eigenspaces. Collecting the $\xi^{(q)}$-terms for $q \in \{2,\dots,k\}$ yields
\begin{equation} \label{eq:Hqequation}
    \bigg( - \frac{\partial}{\partial t} + A^{\odot \star}\bigg) j(H_q(t,\xi^{(q)})) = j(D_\xi H_q(t,\xi^{(q)})M_0 \xi + Q_0 P_q(t,\xi^{(q)})) - R_q(t,\xi^{(q)}),
\end{equation}
where $R_q$ depends on $R_{q'}$ for $2 \leq q' \leq q$, and on $H_{q'}$ and $P_{q'}$ for $2 \leq q' \leq q-1$. For example, the quadratic term $R_2(t,\xi^{(q)})$ is given by $W_2(t,(Q_0 \xi)^{(2)})$.

\textbf{Step 4: Projecting on subspaces.} Our aim is to solve \eqref{eq:Hqequation} for $q \in \{2,\dots,k\}$. To achieve this, let us project this equation onto the (un)stable and center subspace by using the decomposition \eqref{eq:decomposition X hyp}. Therefore, consider the decompositions
\begin{align*}
    H_q(t,\xi^{(q)}) &= Q_0 H_q^0(t,\xi^{(q)}) + H_q^{+}(t,\xi^{(q)}) + H_q^{-}(t,\xi^{(q)}),\\
    R_q(t,\xi^{(q)}) &= j(Q_0 R_q^0(t,\xi^{(q)})) + R_q^{+}(t,\xi^{(q)}) + R_q^{-}(t,\xi^{(q)}),
\end{align*}
where $H_q^0(t,\xi^{(q)}), R_q^0(t,\xi^{(q)}) \in \mathbb{R}^{n_0}, H_q^{\pm}(t,\xi^{(q)}) = P_{\pm} H_q(t,\xi^{(q)})$ and $R_q^{\pm}(t,\xi^{(q)}) = P_{\pm}^{\odot \star} R_q(t,\xi^{(q)})$ for all $t \in \mathbb{R}$ and $\xi \in \mathbb{R}^{n_0}$. Substituting these decompositions into \eqref{eq:Hqequation} and recalling the identities from \eqref{eq:xixi2 terms} yields for the left-hand side of this equation
\begin{align} \label{eq:Hq+-}
\begin{split}
    \bigg( - \frac{\partial}{\partial t} + A^{\odot \star}\bigg) j(H_q(t,\xi^{(q)})) &= j\bigg(Q_0 \bigg( - \frac{\partial}{\partial t} + M_0\bigg)H_q^0(t,\xi^{(q)}) \bigg) \\
    &+ \bigg( - \frac{\partial}{\partial t} + A^{\odot \star}\bigg) j(H_q^{-}(t,\xi^{(q)}) + H_q^{+}(t,\xi^{(q)})),
\end{split}
\end{align}
which must be equal to the right-hand side of \eqref{eq:Hqequation}.

\textbf{Step 5: Existence of $H_q^{\pm}$.} With \eqref{eq:Hq+-} in mind, note that \eqref{eq:Hqequation} on $X_{\pm}^{\odot \star}$ is equivalent to
\begin{equation*}
    \bigg( - \frac{\partial}{\partial t} + A^{\odot \star}\bigg) j(H_q^{\pm}(t,\xi^{(q)})) = j(D_\xi H_{q}^{\pm}(t,\xi^{(q)}) M_0 \xi) - R_q^{\pm}(t,\xi^{(q)}).
\end{equation*}
Along the same lines of the proof in \cite[Theorem 24]{Article2}, we obtain using \Cref{hyp:1}, \Cref{hyp:2} and \eqref{eq:Tsunstarj} the following well-defined explicit integral representations
\begin{align*}
    H_q^{+}(t,\xi^{(q)}) &= - \int_t^\infty T(t-\theta)j^{-1}R_q^+(\theta,(e^{(\theta-t) M_0} \xi)^{(q)}) d\theta,\\
    H_q^{-}(t,\xi^{(q)}) &= j^{-1} \int_{-\infty}^t T^{\odot \star}(t-\theta) R_q^-(\theta,(e^{(\theta-t) M_0} \xi)^{(q)}) d\theta.
\end{align*}
The expression for $H_q^{+}$ is obtained from recalling \eqref{eq:Tsunstarj} and must be interpreted as a Riemann integral, while the expression for $H_q^{-}$ must be interpreted as a weak$^\star$ Riemann integral. Similarly as in the proof of \cite[Theorem 24]{Article2}, we obtain that $H_q^{\pm}$ is continuous and $T$-periodic in the first component. 

\textbf{Step 5: Existence of $H_q^{0}$.} Similar to previous step, but now considering \eqref{eq:Hqequation} on $j(X_0)$ yields
\begin{equation*}
    - \frac{\partial}{\partial t} H_q^0(t,\xi^{(q)}) + 
    M_0 H_q^0(t,\xi^{(q)}) -  D_\xi  H_q^0(t,\xi^{(q)})M_0 \xi = P_q(t,\xi^{(q)}) - R_q^0(t,\xi^{(q)}).
\end{equation*}
This is exactly the same equation as obtained in \cite[Theorem 5.2]{Haragus2011} and so the same proof can be followed to obtain \eqref{eq:restrictionP} as a condition for $P$. It is also proven in the mentioned reference that the maps $H_q^0$ and $P_q$ are continuous for each $q=2,\dots,k$, which proves that $H$ and $P$ are $C^k$-smooth.
\end{proof}

From this theorem, we obtain three additional results. First, note that $\mathcal{W}_{\loc}^c(\overline{x})$ can be written locally around $\mathbb{R} \times \{\overline{x}\}$ as
\begin{equation*}
    \mathcal{W}_{\loc}^c(\overline{x}) = \{ (t,\overline{x} + Q_0 \xi + H(t,\xi)) \in \mathbb{R} \times X : \xi \in \mathbb{R}^{n_0} \},
\end{equation*}
and has the same properties as the description of $\mathcal{W}_{\loc}^c(\overline{x})$ given in \eqref{eq:Wclocxbar}. Second, this shows that $\mathcal{W}_{\loc}^c(\overline{x})$ has a $T$-periodic \emph{parametrization} since the map $t \mapsto \overline{x} + Q_0 \xi + H(t,\xi)$ is $T$-periodic for all $\xi \in \mathbb{R}^{n_0}$. Third, a straightforward computation shows that \eqref{eq:restrictionP} is equivalent with
\begin{equation} \label{eq:Normalformxi2}
    D_1 P(t,\xi) = D_2 P(t,\xi)M_0^\star \xi - M_0^\star P(t,\xi), \quad \forall (t,\xi) \in \mathbb{R} \times \mathbb{R}^{n_0},
\end{equation}
which might be more helpful for the derivation of periodically forced normal forms in concrete applications. The reader with knowledge on normal form theory might recognise \eqref{eq:Normalformxi2} as an inner product normal form where the Lie derivative and homological operator are present, see \cite{Murdock2003,Elphick1987,Elphick1987a} for more information.

\section{Spectral computations for periodically forced DDEs} \label{sec:spectral}
To compute the critical normal form coefficients of certain bifurcations in \Cref{sec:normalization} using \eqref{eq:homological}, we require explicit representations of certain linear objects. Specifically, we need explicit expressions for the (generalized) eigenfunctions of $A$ and $A^\star$. In addition, we must be able to analyse the solutions of a particular periodic linear operator equation of the form \eqref{eq:solvability}.

We start with recalling the (generalized) (adjoint) eigenfunctions in the setting of autonomous linear DDEs. These can be expressed explicitly in terms of Jordan chains of the characteristic matrix $\Delta(z)$ from \eqref{eq:Delta}, see \cite[Theorem IV.5.5 \& IV.5.9]{Diekmann1995} and \cite[Proposition 1]{Bosschaert2024a} for a general result. As our focus in \Cref{sec:normalization} is on codim 1 bifurcations of equilibria in nonlinearly periodically forced DDEs, we restrict our attention to the case where $\lambda$ is a simple eigenvalue of $A$ (and $A^\star$).

\begin{proposition} \label{prop:eigenfunctions}
Let $\lambda$ be a simple eigenvalue of $A$. If $q$ is a right null vector of $\Delta(\lambda)$ then $\varphi$ given by
\begin{equation} \label{eq:eigfunction}
    \varphi(\theta) = e^{\lambda \theta} q, \quad \forall \theta \in [-h,0],
\end{equation}
is an eigenfunction of $A$ corresponding to $\lambda$. If $p$ is a left null vector of $\Delta(\lambda)$ then $\varphi^\odot$ given by
\begin{equation} \label{eq:adjointeigfunction}
    \varphi^\odot(\theta) = \bigg( p, \theta \mapsto \int_{\theta}^h e^{\lambda (\theta-s)} d\eta(s) \bigg), \quad \forall \theta \in [0,h],
\end{equation}
is an eigenfunction of $A^\star$ corresponding to $\lambda$. Moreover, the following identity holds
\begin{equation*}
    \langle \varphi^\odot, \varphi \rangle = p \Delta'(\lambda)q \neq 0.
\end{equation*}
\end{proposition}

Next, we turn our attention to solvability of periodic linear operator equations of the form
\begin{equation} \label{eq:solvability}
    (zI - \mathcal{A}^{\odot \star})(v_0,v) = (w_0,w),
\end{equation}
where $z \in \mathbb{C}$, $(w_0,w) \in C_T(\mathbb{R},X^{\odot \star})$ is given, and $(v_0,v) \in \mathcal{D}(\mathcal{A}^{\odot \star})$ is the unknown. In general, both $z$ and the right-hand side of \eqref{eq:solvability} will have a nontrivial imaginary part and so it is necessary to regard
\eqref{eq:solvability} as its complexification. Here, the (complexification of the) linear operator $\mathcal{A}^{\odot \star} : \mathcal{D}(\mathcal{A}^{\odot \star}) \subseteq C_T(\mathbb{R},X^{\odot \star}) \to C_T(\mathbb{R},X^{\odot \star})$ is defined by
\begin{align}
\begin{split} \label{eq:curlyAsunstar}
    \mathcal{D}(\mathcal{A}^{\odot \star}) &\coloneqq \{ \boldsymbol{\varphi}^{\odot \star} \in C_T^1(\mathbb{R},X^{\odot \star}) : \boldsymbol{\varphi}^{\odot \star}(t) \in \mathcal{D}(A^{\odot \star}) \mbox{ for all } t \in \mathbb{R} \}, \\
    (\mathcal{A}^{\odot \star} \boldsymbol{\varphi}^{\odot \star})(t) &\coloneqq A^{\odot \star}\boldsymbol{\varphi}^{\odot \star}(t) - \dot{\boldsymbol{\varphi}}^{\odot \star}(t),
\end{split}
\end{align}
where we recall from \eqref{eq:AsunstarDDE} the domain and action of $A^{\odot \star}$. A natural extension of this operator, where $A^{\odot \star}$ is $T$-periodic rather than autonomous, arises in the analysis of bifurcations of limit cycles in DDEs. As a result, we can mostly draw upon the technical results established in \cite{Article2,Bosschaert2025}. For example, consider the linear operator $\mathcal{A}^\odot : \mathcal{D}(\mathcal{A}^\odot) \subseteq C_T(\mathbb{R}, X^\odot) \to C_T(\mathbb{R}, X^\odot)$ given by
\begin{align*}
    \mathcal{D}(\mathcal{A}^\odot) &\coloneqq \left\{ \boldsymbol{\varphi}^\odot \in C_T^1(\mathbb{R}, X^\odot) : \boldsymbol{\varphi}^\odot(t) \in \mathcal{D}(A^\odot) \text{ for all } t \in \mathbb{R} \right\}, \\
    (\mathcal{A}^\odot \boldsymbol{\varphi}^\odot)(t) &\coloneqq A^\odot\boldsymbol{\varphi}^\odot(t) + \dot{\boldsymbol{\varphi}}^\odot(t),
\end{align*}
then $\mathcal{A}^{\odot \star}$ is the unique adjoint of $\mathcal{A}^\odot$ with respect to the (complexified) bilinear pairing $\langle \cdot, \cdot \rangle_T : C_T(\mathbb{R}, X^{\odot \star}) \times C_T(\mathbb{R}, X^\odot) \to \mathbb{C}$ defined by
\begin{equation} \label{eq:pairingT}
    \langle \boldsymbol{\varphi}^{\odot \star}, \boldsymbol{\varphi}^\odot \rangle_T \coloneqq \frac{1}{T} \int_0^T \langle \boldsymbol{\varphi}^{\odot \star}(t), \boldsymbol{\varphi}^\odot(t) \rangle \, dt, \quad \forall (\boldsymbol{\varphi}^{\odot \star}, \boldsymbol{\varphi}^\odot) \in C_T(\mathbb{R}, X^{\odot \star}) \times C_T(\mathbb{R}, X^\odot).
\end{equation}
Analogously, one can introduce the linear operator $\mathcal{A}$ and its uniquely defined adjoint $\mathcal{A}^\star$ in terms of the nondegenerate bilinear pairing from \eqref{eq:pairingT}, but now defined on the product space $C_T(\mathbb{R}, X^{\star}) \times C_T(\mathbb{R}, X)$. Furthermore, the embedding $j$ can be lifted to a mapping $\iota : C_T(\mathbb{R}, X) \to C_T(\mathbb{R}, X^{\odot \star})$, defined by $(\iota \boldsymbol{\varphi})(t) \coloneqq j\boldsymbol{\varphi}(t)$ for all $t \in \mathbb{R}$. Due to the $\odot$-reflexivity of $X$ with respect to $T_0$, we know that $\iota \boldsymbol{\varphi}$ takes values in $X^{\odot \odot}$. In what follows, we will frequently write $\iota \varphi$ for a function $\varphi \in X$, where $\varphi$ must then be understood as a constant function in $C_T(\mathbb{R},X)$.

From the results in \cite[Section 3]{Article2}, it turns out that $\mathcal{A}$, $\mathcal{A}^\star$, $\mathcal{A}^{\odot}$, and $\mathcal{A}^{\odot \star}$ are non-closed but closable linear operators. Consequently, the standard notion of the spectrum for closed linear operators used earlier in this article (\Cref{thm:spectralgab}) is not applicable in this context. We thus recall the generalized definition of spectra for closable linear operators used in \cite{Article2,Bosschaert2025,Taylor1986}, which coincides with the standard definition of closed linear operators when the linear operator $A$ in \Cref{def:spectrumclosable} is itself closed.

\begin{definition} \label{def:spectrumclosable}
Let $Y$ be a complex Banach space and $J : \mathcal{D}(J) \subseteq Y \to Y$ a closable linear operator. A complex number $z$ belongs to the \emph{resolvent set} $\rho(J)$ of $J$ if the operator $zI - J$ is injective, has dense range $\mathcal{R}(zI-J) \supseteq \mathcal{D}(J)$, and the \emph{resolvent} of $J$ at $z$ 
defined by $(zI-J)^{-1}: \mathcal{R}(zI-J) \to \mathcal{D}(J)$ is a bounded linear operator. The \emph{spectrum} $\sigma(J)$ of $J$ is defined to be the complement of $\rho(J)$ in $\mathbb{C}$, and the \emph{point spectrum} $\sigma_p(J) \subseteq \sigma(J)$ of $J$ is the set of those $\sigma \in \mathbb{C}$ such that $\sigma I - J$ is not injective. \hfill $\lozenge$
\end{definition}

Having developed a complete understanding of the linear operators involved in \eqref{eq:solvability}, we now turn our attention to the question of solvability. From \eqref{eq:AsunstarDDE} and \eqref{eq:curlyAsunstar}, it follows that \eqref{eq:solvability} can be interpreted as a periodic linear operator equation on the product space $C_T(\mathbb{R},\mathbb{C}^n) \times C_T(\mathbb{R},L^{\infty}([-h,0],\mathbb{C}^n)) \cong C_T(\mathbb{R},X^{\odot \star})$. To analyse the first component of this equation, some preparatory steps are required. Therefore, let us introduce \emph{characteristic operator} $\boldsymbol{\Delta}(z) : \mathcal{D}(\boldsymbol{\Delta}(z)) \coloneqq C_T^1(\mathbb{R},\mathbb{C}^n) \subseteq C_T(\mathbb{R},\mathbb{C}^n) \to C_T(\mathbb{R},\mathbb{C}^n)$ defined by
\begin{equation} \label{eq:Delta(z)q}
    (\boldsymbol{\Delta}(z)\boldsymbol{q})(t) \coloneqq \dot{\boldsymbol{q}}(t) + z \boldsymbol{q}(t) - \int_0^h d \zeta (\theta) e^{-z \theta} \boldsymbol{q}(t - \theta), \quad \forall \boldsymbol{q} \in C_T^1(\mathbb{R},\mathbb{C}^n),
\end{equation}
which can be interpreted as a $T$-periodic extension of the characteristic matrix $\Delta(z)$ from \eqref{eq:Delta}. Spectral properties of a generalization of this linear operator, where $\zeta$ is $T$-periodic in an additional component, have been extensively studied in \cite[Section 3]{Bosschaert2025}. The following result will be helpful in the proof of the resolvent representation for \eqref{eq:solvability} given in \Cref{prop:resolvent}, where we used the shorthand notation $\omega_T \coloneqq 2 \pi / T$ as this frequency will also play an important role in \Cref{subsec:hopf}.

\begin{lemma} \label{lemma:characoperator}
The linear operator $\boldsymbol{\Delta}(z)$ is closed and there holds
\begin{equation}  \label{eq:charoperatorspectra}
    \{ z \in \mathbb{C} : \boldsymbol{\Delta}(z) \text{ is not invertible} \} \hspace{-2pt} = \hspace{-2pt} \{ z \in \mathbb{C} : \det \Delta(z+im\omega_T) = 0 \text{ for some } m \in \mathbb{Z}  \} \hspace{-2pt} = \hspace{-2pt} \sigma(A) + i\omega_T \mathbb{Z}.  
\end{equation}
\end{lemma}
\begin{proof}
The claim on the closedness of $\boldsymbol{\Delta}(z)$ can be proven as in \cite[Lemma 3.2]{Bosschaert2025}. For convenience in what follows, we denote the period by $\tilde{T}$ instead of $T$ in this proof. To prove \eqref{eq:charoperatorspectra}, consider the equation $\boldsymbol{\Delta}(z)\boldsymbol{q} = \boldsymbol{f}$ for some given $\boldsymbol f \in C_{\tilde{T}}(\mathbb{R},\mathbb{C}^n)$ and note that this is equivalent to
\begin{equation} \label{eq:wBVP}
    \boldsymbol{\dot{w}}(t) = L \boldsymbol{w}_t + e^{zt}\boldsymbol{f}(t), \quad \boldsymbol{w}_{\tilde{T}} = e^{z\tilde{T}} \boldsymbol{w}_0,
\end{equation}
where $\boldsymbol{w}(t) = e^{zt} \boldsymbol{q}(t)$ for all $t \in \mathbb{R}$. Recalling the results from \cite[Chapter VI]{Diekmann1995}, the unique solution to this problem is of the form
\begin{equation} \label{eq:w_t}
    \boldsymbol{w}_t = T(t)\eta_0 + j^{-1} \int_0^t T^{\odot \star}(t-s)  e^{zs} \boldsymbol{f}(s)r^{\odot \star} ds, \quad \forall t \in [0,\tilde{T}],
\end{equation}
where $\eta_0 \in X$ must satisfy
\begin{equation*}
    (e^{z\tilde{T}}I - T(\tilde{T}))\eta_0 =  j^{-1} \int_0^{\tilde{T}} T^{\odot \star}(\tilde{T}-s) e^{z s} \boldsymbol{f}(s)r^{\odot \star} ds.
\end{equation*}
Since $T$ is an eventually compact $\mathcal{C}_0$-semigroup, the spectral mapping theorem \cite[Corollary IV.3.12]{Engel2000} tells us that this equation has a unique solution if and only if $\text{exp}({z\tilde{T}}) \notin \text{exp}({\tilde{T}\sigma(A)})$, which is equivalent to saying that $z \notin \sigma(A) +  i \omega_{\tilde{T}} \mathbb{Z}$. It remains to show that $\boldsymbol{q}$ is $C^1$-smooth. Let $I : [0,\tilde{T}] \to X : t \mapsto I(t)$ denote the weak$^\star$ integral appearing in \eqref{eq:w_t}. Recall from \cite[Lemma III.2.1]{Diekmann1995} that the map $I$ is continuous since $\boldsymbol{f}$ is continuous. Since $j^{-1}$ is continuous and $T$ is strongly continuous, it follows from \eqref{eq:w_t} that the map $t \mapsto \boldsymbol{w}_t$ is continuous on $[0,\tilde{T}]$. Since $L$ and $\boldsymbol{f}$ are continuous, the map $t \mapsto L \boldsymbol{w}_t + \boldsymbol{f}(t)$ is continuous on $[0,\tilde{T}]$. Hence, \eqref{eq:wBVP} tells us that $\boldsymbol{w}$, and thus $\boldsymbol{q}$, are $C^1$-smooth on $[0,\tilde{T}]$. The function $\boldsymbol{q}$ is now obtained from a trivial $\tilde{T}$-periodic extension. A straightforward computation using \eqref{eq:wBVP} shows that $\dot{\boldsymbol{q}}(\tilde{T}) = \dot{\boldsymbol{q}}(0)$, meaning that $\boldsymbol{q}$ is $C^1$-smooth on $\mathbb{R}$. This proves the result as the second equality in \eqref{eq:charoperatorspectra} is an immediate consequence of \eqref{eq:spectraDelta}.
\end{proof}

\begin{remark} \label{remark:FourierDelta1}
To gain a more intuitive understanding of the first equality appearing in \eqref{eq:charoperatorspectra}, consider the unique solution $\boldsymbol{q} \in C_T^1(\mathbb{R}, \mathbb{C}^n)$ to the equation $\boldsymbol{\Delta}(z)\boldsymbol{q} = \boldsymbol{f}$ for a given complex number $z \notin \sigma(A) + i \omega_T \mathbb{Z}$ and (in)homogeneity $\boldsymbol{f} \in C_T(\mathbb{R}, \mathbb{C}^n)$, along with their respective (formal) Fourier expansions:
\begin{equation*}
    \boldsymbol{q}(t) = \sum_{m \in \mathbb{Z}} q_m e^{i m \omega_T t}, \quad 
    \boldsymbol{f}(t) = \sum_{m \in \mathbb{Z}} f_m e^{i m \omega_T t}, \quad 
    \forall t \in \mathbb{R}.
\end{equation*}
Substituting these expansions into \eqref{eq:Delta(z)q} and recalling the definition of the characteristic matrix $\Delta(z)$ from \eqref{eq:Delta} yields
\begin{equation} \label{eq:qm}
    \Delta(z + i m \omega_T) q_m = f_m, \quad \forall m \in \mathbb{Z},
\end{equation}
which admits a unique solution $q_m = \Delta(z + i m \omega_T)^{-1} f_m$ if and only if $\det \Delta(z + i m \omega_T) \neq 0$ for all $m \in \mathbb{Z}$. However, a direct proof of \eqref{eq:charoperatorspectra} using this approach is rather challenging, as it requires sharp estimates of $\|\Delta(z + i m \omega_T)^{-1}\|$ as $m \to \pm \infty$ in order to ensure that $\boldsymbol{q}$ is continuously differentiable. \hfill $\lozenge$
\end{remark}

The spectral equality in \eqref{eq:charoperatorspectra} enables us to derive a concrete representation of the resolvent operator of $\mathcal{A}^{\odot \star}$ as stated in the following result.

\begin{proposition} \label{prop:resolvent}
If $z \notin \sigma(A) + i \omega_T \mathbb{Z}$, then the resolvent of $\mathcal{A}^{\odot \star}$ at $z$ reads
\begin{equation*}
    (v_0,v) = (zI - \mathcal{A}^{\odot \star})^{-1} (w_0,w),
\end{equation*}
where
\begin{equation*}
    v(t)(\theta) = e^{z \theta} v_0(t + \theta) + \int_\theta^0 e^{z(\theta -s)} w(t + \theta - s)(s) ds, \quad \forall t \in \mathbb{R}, \ \theta \in [-h,0],
\end{equation*}
and $v_0$ is the unique solution of the periodic linear (inhomogeneous) DDE
\begin{equation*}
    v_0 = \boldsymbol{\Delta}(z)^{-1}\bigg( t \mapsto w_0(t) + \int_0^h d \zeta(\theta)\int_{-\theta}^{0} e^{-z(\theta + s)} w(t -\theta - s)(s) ds \bigg).
\end{equation*}
\end{proposition}
\begin{proof}
The proof is similar to that of \cite[Proposition 3.10]{Bosschaert2025} and therefore omitted.
\end{proof}
To identify the spectrum of $\mathcal{A}, \mathcal{A}^\star, \mathcal{A}^\odot$ and $\mathcal{A}^{\odot \star}$, we proceed along the same lines of \cite[Section 3.2]{Bosschaert2025} and arrive at the following result.

\begin{proposition} \label{prop:spectraequal}
The spectra of $\mathcal{A}, \mathcal{A}^\star, \mathcal{A}^\odot$ and $\mathcal{A}^{\odot \star}$ are all equal, consist of point spectrum only, and there holds
\begin{equation*}
    \sigma(\mathcal{A}^{\odot \star}) = \{ z \in \mathbb{C} : \boldsymbol{\Delta}(z) \mbox{ is not invertible} \} =\sigma(A) + i \omega_T \mathbb{Z}.  
\end{equation*}
\end{proposition}

Solvability of \eqref{eq:solvability} naturally leads to a distinction between two cases, depending on whether or not the complex number $z$ lies in the set $\sigma(A) + i \omega_T \mathbb{Z}$. If $z \notin \sigma(A) + i \omega_T \mathbb{Z}$, then we have shown in \Cref{prop:resolvent} that \eqref{eq:solvability} admits a unique solution $(v_0,v)$ in the domain of $\mathcal{A}^{\odot \star}$. As will become clear in \Cref{sec:normalization}, an explicit representation of the resolvent operator is required for a specific right-hand side $(w_0, w)$. 

\begin{corollary} \label{cor:resolvent} 
Suppose that $z \notin \sigma(A) + i \omega_T \mathbb{Z}$. If $(w_0,w) = (t \mapsto w_0(t)r^{\odot \star})$, then the unique solution $(v_0,v)$ of \eqref{eq:solvability} reads
\begin{equation*}
    v(t)(\theta) = e^{z \theta} v_0(t + \theta), \quad \forall t \in \mathbb{R}, \ \theta \in [-h,0], \quad v_0 = \boldsymbol{\Delta}^{-1}(z)w_0.
    \end{equation*}
Furthermore, if $w_0^m$ denote the Fourier coefficients of $w_0$, then $v_0$ is given by
\begin{equation} \label{eq:v0w0Fourier}
    v_0(t) = \sum_{m \in \mathbb{Z}} v_0^m e^{im \omega_T t}, \quad v_0^m \coloneqq \Delta(z+ i m \omega_T)^{-1} w_0^m.
\end{equation}
\end{corollary}

\begin{proof}
The first part of the proof follows directly from \Cref{prop:resolvent} while the explicit formulas from second part of the proof can be obtained as in \Cref{remark:FourierDelta1}.
\end{proof}

Let us now turn our attention to the solvability of \eqref{eq:solvability} when $z = \lambda \in \sigma(A) + i \omega_T \mathbb{Z}$. Then \eqref{eq:solvability} need not to have a solution. To find a solution, we will use a Fredholm alternative that is suited for periodic linear operator equations of the form \eqref{eq:solvability}. Before we can prove that such a solvability condition holds, we first need to find an adjoint $\boldsymbol{\Delta}^\star(z) \coloneqq \boldsymbol{\Delta}(z)^\star$ of $\boldsymbol{\Delta}(z)$, with respect to the (complexified) bilinear \emph{pairing} $\langle \cdot, \cdot \rangle_T : C_T(\mathbb{R}, \mathbb{C}^{n \star}) \times C_T(\mathbb{R}, \mathbb{C}^n) \to \mathbb{C}$ defined by
\begin{equation} \label{eq:pairingT2}
    \langle \boldsymbol p, \boldsymbol q \rangle_T \coloneqq \frac{1}{T}\int_0^T \boldsymbol p(t) \boldsymbol q(t) dt, \quad \forall (\boldsymbol p,\boldsymbol q) \in C_T(\mathbb{R}, \mathbb{C}^{n \star}) \times C_T(\mathbb{R}, \mathbb{C}^n).
\end{equation}
The following result states that such an adjoint exists and inherits the same properties as $\boldsymbol{\Delta}(z)$. The proof of the following result is similar to that of \cite[Lemma 3.5]{Bosschaert2025} and therefore omitted.

\begin{proposition} \label{prop:adjoint charac}
For any $z \in \mathbb{{C}}$, the operator $\boldsymbol{\Delta}^\star(z) : \mathcal{D}(\boldsymbol{\Delta}^\star(z)) \coloneqq C_T^1(\mathbb{R},\mathbb{C}^{n \star}) \subseteq C_T(\mathbb{R},\mathbb{C}^{n \star}) \to C_T(\mathbb{{R}},\mathbb{C}^{n \star})$ defined by
\begin{equation} \label{eq:Delta(z)starp}
    (\boldsymbol{\Delta}^\star(z) \boldsymbol p)(t) \coloneqq -\dot{\boldsymbol p}(t) + z \boldsymbol p(t) - \int_0^h   \boldsymbol p(t + \theta) e^{-z \theta} d \zeta(\theta), \quad \forall \textbf{p} \in C_T^1(\mathbb{R},\mathbb{C}^{n \star}), 
\end{equation}
is the unique linear operator satisfying
\begin{equation*}
    \langle \boldsymbol{\Delta}^\star(z) \boldsymbol p, \boldsymbol q \rangle_T = \langle \boldsymbol p,\boldsymbol{\Delta}(z) \boldsymbol q \rangle_T, \quad \forall (\boldsymbol p, \boldsymbol q) \in C_T^1(\mathbb{R},\mathbb{C}^{n \star}) \times C_T^1(\mathbb{R},\mathbb{C}^n).
\end{equation*}
Moreover, $\boldsymbol{\Delta}^\star(z)$ is closed and there holds
\begin{equation} \label{eq:boldDeltaDeltastar}
    \{ z \in \mathbb{C} : \boldsymbol{\Delta}^\star(z) \mbox{ is not invertible} \} = \{ z \in \mathbb{C} : \boldsymbol{\Delta}(z) \mbox{ is not invertible} \}.
\end{equation}
\end{proposition}

\begin{remark}
To gain a more intuitive understanding of \eqref{eq:boldDeltaDeltastar}, consider the unique solution $\boldsymbol{p} \in C_T^1(\mathbb{R}, \mathbb{C}^{n\star})$ to the equation $\boldsymbol{\Delta}^\star(z) \boldsymbol{p} = \boldsymbol{f}$ for a given $z \notin \sigma(A) + i \omega_T \mathbb{Z}$ and $\boldsymbol{f} \in C_T(\mathbb{R}, \mathbb{C}^{n \star})$, along with their respective Fourier expansions. Substituting these expansions into \eqref{eq:Delta(z)starp} and recalling the definition of the characteristic matrix $\Delta(z)$ from \eqref{eq:Delta} yields
\begin{equation} \label{eq:pm}
    p_m\Delta(z-im \omega_T) = f_m, \quad \forall m \in \mathbb{Z},
\end{equation}
where $p_m$ and $f_m$ denote the Fourier coefficients of $\boldsymbol p$ and $\boldsymbol f$, respectively. This equation admits a unique solution $p_m = f_m \Delta(z-im \omega_T)^{-1}$ if and only if $\det \Delta(z-im\omega_T) \neq 0$ for all $m \in \mathbb{Z}$. \hfill $\lozenge$    
\end{remark}

To apply in \Cref{sec:normalization} the upcoming Fredholm alternative suited for periodic linear operator equations of the form \eqref{eq:solvability}, it is convenient to have an explicit representation available of the (adjoint) (generalized) eigenfunctions of $\mathcal{A}$ and $\mathcal{A}^\star$ for which we already recall \Cref{prop:spectraequal}. The following result provides such an explicit representation in the case of a \emph{simple eigenvalue} $\lambda$ of $\mathcal{A}$, meaning that $\lambda = \lambda_A + i m\omega_T$ for some simple eigenvalue $\lambda_A$ of $A$ and $m \in \mathbb{Z}$. One can interpret the following result as a natural extension of \Cref{prop:eigenfunctions} in the $T$-periodic setting.

\begin{proposition} \label{prop:curlyeigenfunctions}
Let $\lambda \in \sigma(A) + i \omega_T \mathbb{Z}$ be simple. If $\boldsymbol q$ is a null function of $\boldsymbol\Delta(\lambda)$ then $\boldsymbol \varphi$ given by
\begin{equation*}
    \boldsymbol\varphi(t)(\theta) = e^{\lambda \theta} \boldsymbol q(t+\theta), \quad \forall t \in \mathbb{R}, \ \theta \in [-h,0],
\end{equation*}
is an eigenfunction of $\mathcal{A}$ corresponding to $\lambda$. If $\boldsymbol p$ is a null function of $\boldsymbol \Delta^\star(\lambda)$ then $\boldsymbol \varphi^\odot$ given by
\begin{equation*}
    \boldsymbol \varphi^\odot(t) = \bigg( \boldsymbol p(t), \theta \mapsto \int_\theta^h e^{\lambda(\theta-s)} \boldsymbol p(t+s-\theta) d\zeta(s)\bigg), \quad \forall t \in \mathbb{R},
\end{equation*}
is an adjoint eigenfunction of $\mathcal{A}^\star$ corresponding to $\lambda$. Furthermore, if $q_m$ and $p_m$ denote the Fourier coefficients of $\boldsymbol q$ and $ \boldsymbol p$, then
\begin{equation} \label{eq:Deltapmqm}
    \Delta(\lambda+im\omega_T)q_m = 0, \quad p_m \Delta(\lambda-im\omega_T) = 0, \quad \forall m \in \mathbb{Z},
\end{equation}
and the following identities hold
\begin{equation} \label{eq:pairingidentity}
    \langle \boldsymbol\varphi^\odot,\boldsymbol\varphi \rangle_T = \langle \boldsymbol p, \boldsymbol\Delta'(\lambda) \boldsymbol q \rangle_T= \sum_{m \in \mathbb{Z}} p_m \Delta'(\lambda - im \omega_T)q_{-m} \neq 0.
\end{equation}
\end{proposition}
\begin{proof}
The representation of the (adjoint) eigenfunction can be obtained similarly as in the proof of \cite[Theorem 5 and 8]{Bosschaert2025}. The condition on the Fourier coefficients $q_m$ and $p_m$ can be verified as in \eqref{eq:qm} and \eqref{eq:pm}. The first and last identity in \eqref{eq:pairingidentity} can be proven similarly as in the proof of \cite[Theorem 5 and 8]{Bosschaert2025}. It only remains to show the second identity in \eqref{eq:pairingidentity}. Using \eqref{eq:Delta(z)q} and the Fourier expansions of $\boldsymbol{q}$ and $\boldsymbol{p}$, a straightforward computation shows that
\begin{equation*}
    \langle \boldsymbol p, \boldsymbol\Delta'(\lambda) \boldsymbol q \rangle_T = \frac{1}{T} \sum_{k,l \in \mathbb{Z}} \int_0^T  p_k \Delta(\lambda + i l\omega_T)q_l e^{i(k+l) \omega_T t} dt = \frac{1}{T}\sum_{m \in \mathbb{Z}}\int_0^T  r_m e^{i m \omega_T t} dt = r_0,
\end{equation*}
where we used in the first equality the uniform convergence of both Fourier series since $\boldsymbol{p}$ and $\boldsymbol{q}$ are $C^1$-smooth, and in the second equality the product formula for Fourier series. Here, $r_m$ denotes the well-defined $m$th term of the convolution of the sequences $(p_m)_m$ and $(q_m)_m$, which is justified by the fact that both sequences belong to $\ell^1(\mathbb{Z})$.
\end{proof}

Recall from \Cref{sec:DDEs} that $\sigma(A)$ contains only finitely many elements within any vertical strip of the complex plane. Consequently, the matrix $\Delta(\lambda + i m \omega_T)$ is singular for only finitely many integers $m \in \mathbb{Z}$. Therefore, the vectors $q_m$ and $p_m$ in \eqref{eq:Deltapmqm} are nonzero for only finitely many such $m$, implying that the sum in \eqref{eq:pairingidentity} is actually finite. In particular, when deriving the critical normal form coefficients for the periodically forced fold and nonresonant Hopf bifurcation in \Cref{sec:normalization}, it turns out that only $q_0$ and $p_0$ are nonzero, meaning that \Cref{prop:curlyeigenfunctions} reduces eventually to \Cref{prop:eigenfunctions} for these cases. We now state the crucial Fredholm alternative corresponding to \eqref{eq:solvability}.

\begin{proposition}[Fredholm solvability condition] \label{prop:FSC}
Suppose that $z=\lambda \in \sigma(A) + i \omega_T \mathbb{Z}$. If $(v_0,v)$ is a solution of \eqref{eq:solvability} then
\begin{equation} \label{eq:FSC} \tag{FSC}
    \langle (w_0,w), \boldsymbol{\varphi}^\odot \rangle_T = 0,
\end{equation}
where $\boldsymbol{\varphi}^\odot$ is an eigenfunction of $\mathcal{A}^\star$ corresponding to $\lambda$. 

\end{proposition}
\begin{proof}
It follows from \eqref{eq:solvability} and \eqref{eq:pairingT} that $\langle (w_0,w),\boldsymbol\varphi^\odot \rangle_T = \langle (\lambda I-\mathcal{A}^{\odot \star})(v_0,v), \boldsymbol\varphi^\odot \rangle_T = \langle (v_0,v), (\lambda I - \mathcal{A}^{\odot}) \boldsymbol\varphi^\odot \rangle_T = 0$, since $\boldsymbol{\varphi}^\odot$ satisfies $\mathcal{A}^\odot \boldsymbol\varphi^\odot = \lambda \boldsymbol\varphi^\odot$ due to \Cref{prop:spectraequal}.
\end{proof}

If $z = \lambda \in \sigma(A) + i \omega_T \mathbb{Z}$ and \eqref{eq:solvability} is consistent, then any solution that exists is not unique as one may add any linear combination of eigenfunctions of $\mathcal{A}^{\odot \star}$. To systematically select a particular solution, we employ the \emph{bordered operator inverse} $(\lambda I - \mathcal{A}^{\odot \star})^{\inv} : \mathcal{R}(\lambda I - \mathcal{A}^{\odot \star}) \to \mathcal{D}(\mathcal{A}^{\odot \star})$. When $\lambda$ is simple, this operator yields a unique solution to the extended system
\begin{equation} \label{eq:borderedinverse}
    (\lambda I - \mathcal{A}^{\odot \star}) (v_0,v) = (w_0,w), \quad \langle (v_0,v), \boldsymbol{\varphi}^\odot \rangle_T = 0,
\end{equation}
for every right-hand side $(w_0,w)$ for which \eqref{eq:solvability} is consistent. Here, $\boldsymbol{\varphi}^\odot$ denotes an adjoint eigenfunction of $\mathcal{A}^\star$ corresponding to the eigenvalue $\lambda$. The pairing in \eqref{eq:borderedinverse} can be computed explicitly in specific settings, see \cite[Proposition 3.15]{Bosschaert2025} for a general result in the setting of bifurcations of limit cycles in DDEs. Since we are mainly interested in \Cref{sec:normalization} in codim 1 bifurcations of equilibria, we do not need such an explicit representation here.

\section{Normalization for periodically forced DDEs} \label{sec:normalization}
Our next objective is to derive explicit computational formulas for the critical normal form coefficients of the periodically forced fold and nonresonant Hopf bifurcation using \eqref{eq:homological}. Although our derivation focuses specifically on these two cases, the normalization method we present below is more general and can also be applied to bifurcations of higher codimension.

According to \Cref{thm:periodicnormalform}, we can assume that a parametrization of (the second component of) $\mathcal{W}_{\loc}^c(\overline{x})$ has been selected so that the restriction of \eqref{eq:DDE} on this invariant manifold has the periodically forced normal form \eqref{eq:Normalformxi}. The same theorem tells us that there exists a locally defined $C^k$-smooth parametrization $\mathcal{H} : \mathbb{R} \times \mathbb{R}^{n_0} \to X$ of (the second component of) $\mathcal{W}_{\loc}^c(\overline{x})$ such that the solution $u(t) = x_t$ of \eqref{eq:DDEODE} can be written as
\begin{equation} \label{eq:u=H}
    u(t) = \mathcal{H}(t,\xi(t)), \quad \forall t \in I,
\end{equation}
where the map $\mathcal{H}$ has the representation
\begin{equation*}
    \mathcal{H}(t,\xi) = \overline{x} + \sum_{i=1}^{n_0} \xi_i \varphi_i + \sum_{|\mu| = 2}^{k} \frac{1}{\mu !} H_\mu(t)\xi^{\mu} + \mathcal{O}(|\xi|^{k+1}).
\end{equation*}
Combining \eqref{eq:DDEODE} and \eqref{eq:u=H} yields the \emph{homological equation}
\begin{equation} \label{eq:homological} \tag{HOM}
    j\bigg( \frac{\partial \mathcal{H}(t,\xi)}{\partial t}  +  \frac{\partial \mathcal{H}(t,\xi)}{\partial \xi}  \dot{\xi} \bigg) = A_0^{\odot \star} j(\mathcal{H}(t,\xi)) + W(t,\mathcal{H}(t,\xi)),
\end{equation}
where $\dot{\xi}$ is given by the periodically forced normal form \eqref{eq:Normalformxi}. The unknowns in \eqref{eq:homological} are $H_\mu$ and the coefficients $P_\mu$, arising from the Taylor expansion of $P(t,\cdot)$, for $2 \leq |\mu| \leq k$. These coefficients are to be computed up to a finite order, depending on the bifurcation of interest. To do so, we first expand $W$ (and thus $F$) as performed in \eqref{eq:TaylorexpansionW}. Next, we collect terms of the form $\xi^\mu$ in ascending order of $|\mu| \geq 2$ and solve the resulting periodic linear operator equations as discussed in \Cref{sec:spectral}. The coefficients $H_\mu$ and $P_\mu$ can then be determined recursively by applying the Fredholm solvability condition \eqref{eq:FSC}.

This method will be used in the following two subsections to derive the critical normal form coefficients of the periodically forced fold (\Cref{subsec:fold}) and the periodically forced nonresonant Hopf bifurcation (\Cref{subsec:hopf}). It is sufficient to expand the nonlinearity $W$ (and $F$) around $(t,\overline{x})$ in the second component up to second order for the fold, and to the third order for the nonresonant Hopf bifurcation. Therefore, it is sufficient to take $k \geq 2$ for the fold, and $k \geq 3$ for the nonresonant Hopf bifurcation. For simplicity of notation, we will also write
\begin{gather*}
    B(t;\psi(t),\psi(t)) \coloneqq D_2^2F(t,\overline{x})(\psi(t)^{(2)}),  \quad C(t;\psi(t),\psi(t),\psi(t)) \coloneqq D_2^3F(t,\overline{x})(\psi(t)^{(3)}), \\
    B(\psi_1,\psi_2) \coloneqq (t \mapsto B(t;\psi_1(t),\psi_2(t)), \quad C(\psi_1,\psi_2,\psi_3) \coloneqq (t \mapsto C(\psi_1(t),\psi_2(t),\psi_3(t))).
\end{gather*}

\subsection{Periodically forced fold bifurcation} \label{subsec:fold}
Suppose that the equilibrium $\overline{x}$ is nonhyperbolic with $\sigma_0 = \{0\}$, where the eigenvalue $0$ of $A$ is simple. Then, there exist an eigenfunction $\varphi$ and an adjoint eigenfunction $\varphi^\odot$ such that
\begin{equation*}
    A \varphi = 0, \quad A^\star \varphi^\odot = 0, \quad \langle \varphi^\odot,\varphi \rangle = 1.
\end{equation*}
The eigenfunctions are explicitly given by \eqref{eq:eigfunction} and \eqref{eq:adjointeigfunction} with $q \in \mathbb{R}^n$ and $p \in \mathbb{R}^{n \star}$ satisfying
\begin{equation} \label{eq:foldnormalization}
    \Delta(0)q = 0, \quad p \Delta(0) = 0, \quad p\Delta'(0)q = 1.
\end{equation}
Since all ingredients of the linear part are determined, we turn our attention to the normal form. 
\begin{proposition} \label{prop:foldNF}
The periodically forced critical normal form at the fold bifurcation reads
\begin{equation} \label{eq:foldNF}
    \dot{\xi} = b \xi^2 + \mathcal{O}(|\xi|^3),
\end{equation}
where $\xi,b \in \mathbb{R}$ and the $\mathcal{O}$-terms are $T$-periodic. The bifurcation is nondegenerate if $b \neq 0$.
\end{proposition}
\begin{proof}
As $M_0 = 0$, \eqref{eq:restrictionP} tells us that $P(t,\xi) = P(0,\xi)$ and thus $P$ is $t$-independent. Hence, \eqref{eq:foldNF} follows immediately from \eqref{eq:Normalformxi}. As \eqref{eq:foldNF} is similar to the critical normal form at the fold bifurcation in autonomous ODEs, the nondegeneracy condition follows from \cite[Theorem 3.1]{Kuznetsov2023a}.
\end{proof}
To illustrate the resulting dynamics of \eqref{eq:foldNF} near the periodically forced fold bifurcation, it is convenient to introduce an unfolding parameter $\beta \in \mathbb{R}$ into \eqref{eq:foldNF} as
\begin{equation} \label{eq:foldPFexample}
    \dot{\xi}(t) = \beta + b \xi(t)^2 + N(t,\xi(t))\xi(t)^3,
\end{equation}
where $N$ is $T$-periodic in the first component and continuous. If \eqref{eq:DDE} is autonomous, then $N$ is $t$-independent, and the resulting dynamics near the fold bifurcation are well understood, see \cite[Section 3.2]{Kuznetsov2023a} for further details. If $b \neq 0$, then the system has two hyperbolic equilibria $\overline{x}_{\pm}$ of opposite stability for $\beta < 0$ if $b > 0$ (or for $\beta > 0$ if $b < 0$). These equilibria collide at the critical parameter value $\beta = 0$, where the resulting equilibrium $\overline{x} = 0$ becomes nonhyperbolic (fold point). For $\beta > 0$ if $b > 0$ (or for $\beta < 0$ if $b < 0$), the equilibrium disappears entirely. If $b = 0$, then a \emph{cusp} bifurcation generically occurs, see \cite[Section 8.2]{Kuznetsov2023a} for more information. 

In the setting of nonlinearly periodically forced ODEs and DDEs, the bifurcation scenario is slightly different. If $b \neq 0$, then the system has two $T$-periodic hyperbolic limit cycles $\Gamma_{\pm}$ of opposite stability for $\beta < 0$ if $b > 0$ (or for $\beta > 0$ if $b < 0$). These limit cycles collide at the critical parameter value $\beta = 0$ where the resulting equilibrium $\overline{x} = 0$ is nonhyperbolic and undergoes a periodically forced fold bifurcation. For $\beta > 0$ if $b > 0$ (or for $\beta < 0$ if $b < 0$), the equilibrium disappears entirely. This behaviour is also illustrated in \Cref{fig:foldNFexample}.

\begin{figure}[ht]
    \centering
\includegraphics[width=0.95\linewidth]{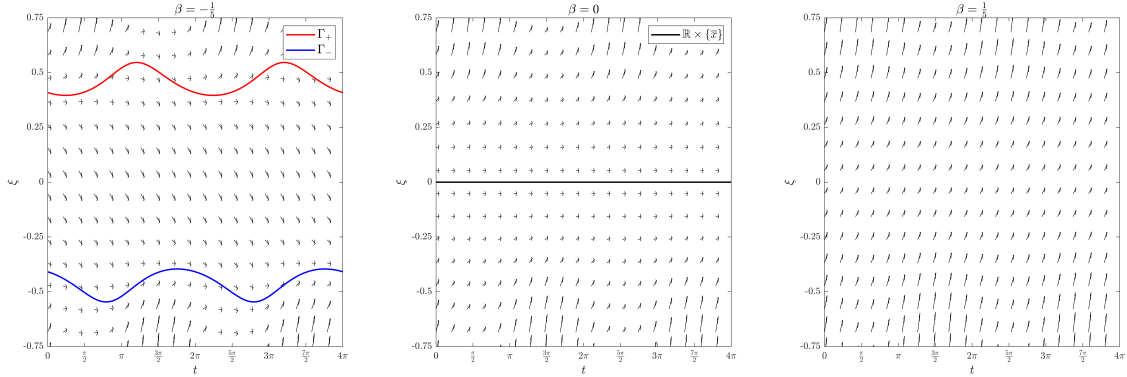}
    \caption{Vector field of the nonlinearly periodically forced ODE \eqref{eq:foldPFexample}, shown together with the invariant sets $\Gamma_{\pm}$ and $\mathbb{R} \times \{\overline{x}\}$ for various values of the unfolding parameter $\beta$, while keeping $b = 1$ and $N(t,\xi(t)) = \sin(t)$ fixed.}
    \label{fig:foldNFexample}
\end{figure}

Let us now derive an explicit computational formula for the critical normal form coefficient $b$ for nonlinearly periodically forced DDEs. The $2$-dimensional center manifold $\mathcal{W}_{\loc}^c(\overline{x})$ at the periodically forced fold bifurcation can be parametrized locally near $\overline{x}$ as
\begin{equation*}
    \mathcal{H}(t,\xi) = \overline{x} + \xi \varphi + \frac{1}{2} H_2(t) \xi^2 + \mathcal{O}(\xi^3), \quad t \in \mathbb{R}, \ \xi \in \mathbb{R},
\end{equation*}
where $H_2$ is $T$-periodic. Filling everything into \eqref{eq:homological}, we obtain for the constant and linear terms the well-known identities \eqref{eq:xixi2 terms}. Collecting the $\xi^2$-terms, we obtain the singular equation
\begin{equation*}
    - \mathcal{A}^{\odot \star} \iota H_{2} = B(\varphi,\varphi)r^{\odot \star} - 2b\iota\varphi,
\end{equation*}
since $0$ is an eigenvalue of $A$, recall \Cref{prop:spectraequal}. 
Applying \eqref{eq:FSC} onto this equation yields
\begin{equation*}
    \langle \boldsymbol{p},B(\varphi,\varphi) \rangle_T - 2 b \langle \boldsymbol{\varphi}^\odot, \varphi \rangle_T =\langle B(\varphi,\varphi)r^{\odot \star} - 2b \iota \varphi, \boldsymbol{\varphi}^\odot \rangle _T = 0,
\end{equation*}
where $\boldsymbol{\varphi}^\odot = (\boldsymbol{p},\boldsymbol{g})$ is an adjoint eigenfunction of $\mathcal{A}^\star$ corresponding to $0$. Since $0$ is the only eigenvalue of $A$ on the imaginary axis, \eqref{eq:Deltapmqm} tells us that the Fourier coefficients $q_m$ and $p_m$ of $\boldsymbol{q}$ and $\boldsymbol{p}$ are given by $q_0 = q$ and $p_0 = p$ while $q_m = p_m = 0$ for all $m \neq 0$. Hence, $\boldsymbol{q}(t) = q$ and $\boldsymbol{p}(t) = p$ for all $t \in \mathbb{R}$. Moreover, we observe that the eigenfunction $\boldsymbol{\varphi}$ of $\mathcal{A}$ corresponding to $0$ is given by $\boldsymbol{\varphi}(t) = \varphi = q$ for all $t \in \mathbb{R}$. An application of \eqref{eq:pairingidentity} and \eqref{eq:foldnormalization} shows that $\langle \boldsymbol\varphi^\odot,\varphi \rangle_T = p \Delta'(0) q = 1$, and so we conclude from \eqref{eq:naturalpairings} and \eqref{eq:pairingT} that
\begin{equation} \label{eq:criticalcoeffffold}
    b = \frac{1}{2}\langle p,B(\varphi,\varphi) \rangle_T.
\end{equation}
\begin{remark} \label{remark:fold}
The expression in \eqref{eq:criticalcoeffffold} naturally extends the formula for the critical normal form coefficient associated with fold bifurcations in autonomous DDEs. In the latter case, where the right-hand side $F$ in \eqref{eq:DDE} does not depend on time, the map $B$ becomes time-independent and the elements involved in the pairing are constant. As a result, we obtain
\begin{equation*}
    b = \frac{1}{2} p B(\varphi,\varphi),
\end{equation*}
which coincides exactly with the expression found in \cite[Equation (3.32)]{Janssens2010}. \hfill $\lozenge$
\end{remark}

\subsection{Periodically forced nonresonant Hopf bifurcation} \label{subsec:hopf}
Suppose that the equilibrium $\overline{x}$ is nonhyperbolic with $\sigma_0 = \{ \pm i \omega\}$, where the eigenvalues $\pm i \omega$ of $A$ are simple and $\omega > 0$. Then, there exist an eigenfunction $\varphi$ and an adjoint eigenfunction $\varphi^\odot$ such that
\begin{equation*}
    A \varphi = i \omega \varphi, \quad A^\star \varphi^\odot = i \omega \varphi^\odot, \quad \langle \varphi^\odot,\varphi \rangle = 1.
\end{equation*}
The eigenfunctions are explicitly given by \eqref{eq:eigfunction} and \eqref{eq:adjointeigfunction} with $q \in \mathbb{C}^n$ and $p \in \mathbb{C}^{n \star}$ satisfying
\begin{equation} \label{eq:hopfnormalization}
    \Delta(i \omega)q = 0, \quad p \Delta(i \omega) = 0, \quad p\Delta'(i \omega)q = 1.
\end{equation}
Since all ingredients of the linear part are determined, we turn our attention to the periodically forced Hopf normal form. In contrast to the Hopf bifurcation in autonomous systems, the periodically forced case exhibits richer behaviour since the structure of $P$ from \eqref{eq:Normalformxi} depends on the rationality of $\omega_T/\omega$. If $\omega_T/\omega \in \mathbb{Q}$, then \eqref{eq:DDE} undergoes a so-called \emph{$\!r:\!s$-resonant Hopf bifurcation}, where $r/s = \omega_T/\omega$ and $r,s \in \mathbb{N}_0$ are irreducible (gcd$(r,s) = 1$). The bifurcation is referred to as \emph{strongly resonant} when $r \in \{1, 2, 3\}$ and as \emph{weakly resonant} when $r \geq 4$, see primarily \cite{Zhang2011,Haragus2011} and secondarily \cite{Gambaudo1985,Bajaj1986,Namachchivaya1987,Vance1991} for further details. On the other hand, if $\omega_T/\omega \notin \mathbb{Q}$, then \eqref{eq:DDE} undergoes a \emph{nonresonant Hopf bifurcation} for which its normal form will be derived below.

\begin{proposition} \label{prop:HopfPF}
The periodically forced critical normal form at the nonresonant Hopf bifurcation reads
\begin{equation} \label{eq:HopfNF}
    \dot{\xi} = i \omega \xi + c \xi |\xi|^2 + \mathcal{O}(|\xi|^4),
\end{equation}
where $\xi,c \in \mathbb{C}$ and the $\mathcal{O}$-terms are $T$-periodic. The bifurcation is nondegenerate if $l_1 \coloneqq \frac{1}{\omega} \Re c \neq 0$.
\end{proposition}
\begin{proof}
As the elements of $\sigma_0$ are purely imaginary, it is convenient cast \eqref{eq:restrictionP} into an ODE in $\mathbb{C}$, where now $M_0 = \mbox{diag}(i \omega, - i \omega)$ and $P$ is a function of $(t,\xi,\bar{\xi})$. Hence, \eqref{eq:restrictionP} shows that
\begin{equation} \label{eq:HopfFourier}
    e^{-i \omega t} P(t,e^{i \omega t} \xi, e^{-i \omega t} \bar{\xi}) = P(0,\xi,\bar{\xi}), \quad \forall (t,\xi) \in \mathbb{R} \times \mathbb{C}.
\end{equation}
Consider the Fourier series of $P(\cdot, \xi,\bar{\xi})$ and let $\alpha_{pq}^{(m)} \xi^p \xi^q e^{im \omega_T t}$ denote the monomials of the associated $m$th Fourier mode. Filling this into \eqref{eq:HopfFourier}, the equation above yields the relation $(p-q-1)\omega + m \omega_T = 0$. Since we are at the nonresonant Hopf bifurcation, there must hold $p = q+1$ and $m=0$, which yields \eqref{eq:HopfNF}. As \eqref{eq:HopfNF} is similar to the critical normal form at the Hopf bifurcation in autonomous ODEs, the nondegeneracy condition follows from \cite[Theorem 3.3]{Kuznetsov2023a}.  
\end{proof}
The real number $l_1$ is called the \emph{first Lyapunov coefficient} and also determines the direction of the bifurcation as will be explained below.

To illustrate the resulting dynamics of \eqref{eq:HopfNF} near the nonresonant Hopf bifurcation, it is convenient to introduce an unfolding parameter $\beta \in \mathbb{R}$ into \eqref{eq:HopfNF} as
\begin{equation} \label{eq:HopfPFexample}
    \dot{\xi}(t) = (\beta+i\omega)\xi(t) + c \xi(t)|\xi(t)|^2 + N(t,\xi(t))|\xi(t)|^4,
\end{equation}
where $N$ is $T$-periodic in the first component and continuous. If \eqref{eq:DDE} is autonomous, then $N$ is $t$-independent, and the resulting dynamics near the Hopf bifurcation is well understood, see \cite[Section 3.2]{Kuznetsov2023a} for further details. If $l_1 < 0$, then the system has a stable hyperbolic trivial equilibrium for $\beta < 0$. This equilibrium becomes nonhyperbolic and weakly attracting at the critical parameter value $\beta = 0$ (Hopf point). For $\beta > 0$, the trivial hyperbolic equilibrium becomes unstable, but is surrounded by a unique stable hyperbolic limit cycle in $\mathbb{C}$. This Hopf bifurcation is called \emph{supercritical}. Conversely, if $l_1 > 0$, the roles of stability are reversed, yielding a \emph{subcritical} Hopf bifurcation with the corresponding unstable limit cycle existing for $\beta < 0$. If $l_1 = 0$, then a \emph{generalized Hopf} (Bautin) bifurcation generically occurs, see \cite[Section 8.3]{Kuznetsov2023a} for more information. 

In the setting of nonlinearly periodically forced ODEs and DDEs, the bifurcation scenario is slightly different \cite{Gambaudo1985,Zhang2011}. If $l_1 < 0$, then the system still has a stable hyperbolic trivial equilibrium for $\beta < 0$, which becomes nonhyperbolic and weakly attracting at $\beta = 0$. For $\beta > 0$, the trivial hyperbolic equilibrium becomes unstable and is surrounded by a stable invariant nonautonomous set $C_\beta$ satisfying $ \mathbb{R} \times \{0\} \subseteq C_\beta \subseteq \mathbb{R} \times \mathbb{C}$. Because the term $N$ in \eqref{eq:HopfPFexample} is $T$-periodic, the invariant set $C_\beta$ becomes generically a helically wound surface that is diffeomorphic to the cylinder $\mathbb{R}\times S^{1}$, and thus appears geometrically as a distorted wavy cylindrical surface. In the absence of periodic forcing, $C_\beta$ reduces to an exact cylinder as explained above. Due to the $T$-periodicity of $N$, it also convenient to identify $C_\beta$ as a 2-torus within $\mathbb{R} / T \mathbb{Z} \times \mathbb{C}$. This nonresonant Hopf bifurcation is also called \emph{supercritical}. When $l_1 > 0$, the roles of stability are reversed, yielding a \emph{subcritical} nonresonant Hopf bifurcation with the corresponding unstable nonautonomous invariant set $C_\beta$ existing for $\beta < 0$. This behaviour is also illustrated in \Cref{fig:HopfNFexample}.

\begin{figure}[ht]
    \centering
    \includegraphics[width=0.95\linewidth]{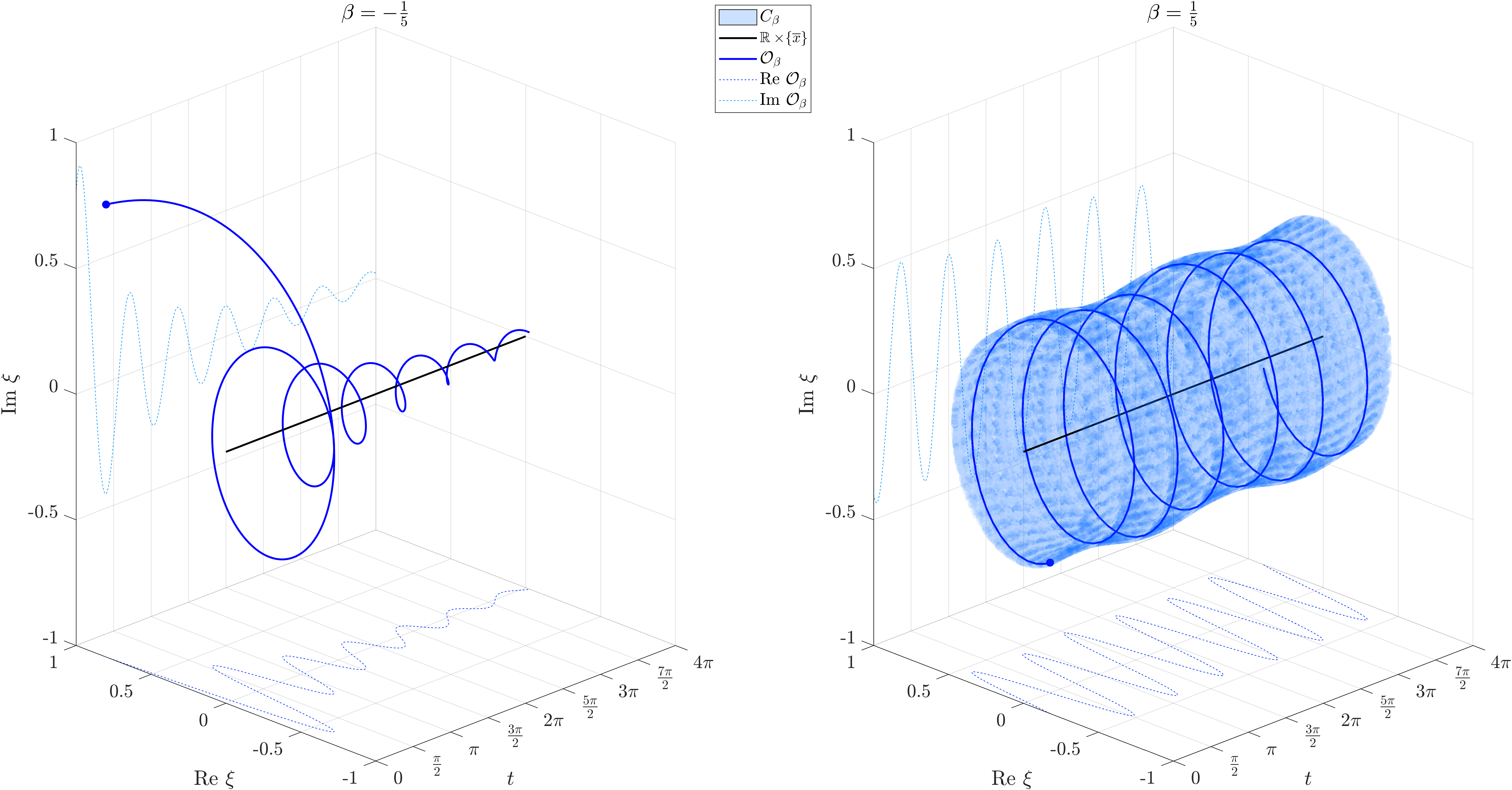}
    \caption{Forward orbits $\mathcal{O}_\beta$ of the nonlinearly periodically forced ODE \eqref{eq:HopfPFexample}, shown together with the invariant sets $C_\beta$ and $\mathbb{R} \times \{\overline{x}\}$ for various values of the unfolding parameter $\beta$, while keeping $\omega = \pi, c=-1$ and $N(t,\xi(t)) = \sin(t)$ fixed.}
    \label{fig:HopfNFexample}
\end{figure}

Let us now derive an explicit computational formula for the critical normal form coefficient $c$ for periodically forced DDEs. The 3-dimensional center manifold $\mathcal{W}_{\loc}^c(\overline{x})$ at the periodically forced nonresonant Hopf bifurcation can be parametrized locally near $\overline{x}$ as
\begin{equation*}
    \mathcal{H}(t,\xi,\bar{\xi}) = \overline{x} + \xi \varphi + \bar{\xi} \overline{\varphi} + \sum_{2 \leq j+k \leq 3} H_{jk}(t) \xi^j \bar{\xi}^k + \mathcal{O}(|\xi|^3), \quad t \in \mathbb{R}, \ \xi \in \mathbb{C},
\end{equation*}
where $H_{jk}$ are $T$-periodic and $H_{ij} = \overline{H_{ji}}$ so that $H_{11}$ is real. Filling everything into \eqref{eq:homological}, we obtain for the constant and linear terms the well-known identities \eqref{eq:xixi2 terms} and their complex conjugates. Collecting the $\xi^2$- and $\xi \bar{\xi}$-terms yield the nonsingular systems
\begin{equation*}
    (2i \omega I - \mathcal{A}^{\odot \star}) \iota H_{20} = B(\varphi,\varphi)r^{\odot \star}, \quad -\mathcal{A}^{\odot \star} \iota H_{11} = B(\varphi,\overline{\varphi})r^{\odot \star},
\end{equation*}
since we are in the setting of a nonresonant Hopf bifurcation, recall \Cref{prop:spectraequal}. Hence, these equations are solved using \Cref{cor:resolvent} to give
\begin{equation*}
    (H_{20})(t)(\theta) = e^{2i\omega \theta} \boldsymbol{\Delta}(2i\omega)^{-1}[B(\varphi,\varphi)](t+\theta) ,\quad (H_{11})(t)(\theta) = \boldsymbol{\Delta}(0)^{-1}[B(\varphi,\varphi)](t+\theta),
\end{equation*}
which can be explicitly computed using the Fourier series given in \eqref{eq:v0w0Fourier}. Collecting the $\xi^2 \bar{\xi}$-terms yields the singular system
\begin{equation*}
    (i\omega I - \mathcal{A}^{\odot \star}) \iota H_{21} = [C(\varphi,\varphi,\overline{\varphi}) + B(\overline{\varphi}, H_{20}) + 2B(\varphi,H_{11})]r^{\odot \star} - 2c \iota \varphi,
\end{equation*}
since $i\omega$ is an eigenvalue of $A$, recall \Cref{prop:spectraequal}. Applying \eqref{eq:FSC} yields eventually
\begin{equation*}
    \langle \boldsymbol{p},C(\varphi,\varphi,\overline{\varphi}) + B(\overline{\varphi}, H_{20}) + 2B(\varphi,H_{11}) \rangle_T - 2 c \langle \boldsymbol{\varphi}^\odot,\varphi \rangle_T = 0,
\end{equation*}
where $\boldsymbol{\varphi}^\odot = (\boldsymbol{p},\boldsymbol{g})$ is an adjoint eigenfunction of $\mathcal{A}^\star$ corresponding to $i \omega$. Since $\pm i \omega$ are the only eigenvalues of $A$ on the imaginary axis, \eqref{eq:Deltapmqm} in combination with the non-resonance condition $\omega_T/\omega \notin \mathbb{Q}$ tells us that the Fourier coefficients $q_m$ and $p_m$ of $\boldsymbol{q}$ and $\boldsymbol{p}$ satisfy $q_0 = q$ and $p_0 = p$ while $q_m = p_m = 0$ for $m \neq 0$. Thus, $\boldsymbol{q}(t)=q$ and $\boldsymbol{p}(t)=p$ for all $t \in \mathbb{R}$. As a consequence, the eigenfunction $\boldsymbol{\varphi}$ of $\mathcal{A}$ corresponding to $i \omega$ is given by $\boldsymbol{\varphi}(t) = \varphi = (\theta \mapsto e^{i \omega \theta}q)$ for all $t \in \mathbb{R}$. An application of \eqref{eq:pairingidentity} and \eqref{eq:hopfnormalization} shows that $\langle \boldsymbol\varphi^\odot, \varphi \rangle_T = p \Delta '(i \omega) q = 1$, and so we conclude from \eqref{eq:naturalpairings} and \eqref{eq:pairingT} that
\begin{equation} \label{eq:criticalcoeffHopf}
    c = \frac{1}{2} \langle p, C(\varphi,\varphi,\overline{\varphi}) + B(\overline{\varphi}, H_{20}) + 2B(\varphi,H_{11}) \rangle_T.
\end{equation}

\begin{remark} \label{remark:Hopf}
The expression in \eqref{eq:criticalcoeffHopf} naturally extends the formula for the critical normal form coefficient associated with Hopf bifurcations in autonomous DDEs In the latter case, where the right-hand side $F$ in \eqref{eq:DDE} does not depend on time, the maps $B$ and $C$ become time-independent and the elements involved in the pairing are constant. As a result, we obtain
\begin{equation} \label{eq:criticalcoeffHopfauto}
    c = \frac{1}{2} p [C(\varphi,\varphi,\overline{\varphi}) + B(\overline{\varphi}, H_{20}) + 2B(\varphi,H_{11})],
\end{equation}
which coincides exactly with the expression found in \cite[Section 6.1.1]{Bosschaert2020}. It is important to note that $\boldsymbol \Delta(z)$ reduces to  $\Delta(z)$ in this context, implying that both $H_{20}$ and $H_{11}$ are also constant in time. \hfill $\lozenge$
\end{remark}

\section{Examples} \label{sec:Examples}
The purpose of this section is to demonstrate how the explicit formulas \eqref{eq:criticalcoeffffold} and \eqref{eq:criticalcoeffHopf} can be employed in concrete examples of nonlinearly periodically forced DDEs. The first example (\Cref{subsec:examplefold}) considers a scalar model system exhibiting a periodically forced fold bifurcation and shows how a periodically forced cusp bifurcation arises as a codim 2 degeneracy. The second example (\Cref{subsec:Wright}) investigates a periodically forced version of Wright's equation, a classical system in delayed dynamics. We show how nonlinear periodic forcing modifies the bifurcation structure by affecting fold and Hopf bifurcations, inducing resonances and changes in the sign of the first Lyapunov coefficient.

\subsection{Periodically forced fold bifurcation in a model system} \label{subsec:examplefold}
Consider the periodically forced scalar DDE
\begin{equation}\label{eq:examplefold}
    \dot{x}(t) = \alpha_1 + \alpha_2 x(t) + g(t,x(t-1)), \quad t \geq 0,
\end{equation}
where $\alpha_{1,2} \in \mathbb{R}$ are parameters and $g \in C^2(\mathbb{R}^2,\mathbb{R})$ satisfies the following conditions: $g(\cdot,0)=0$ and $g$ is $T$-periodic in its first argument but has autonomous linear part in the sense that the map $t \mapsto D_2 g(t,0)$ is constant. Under these assumptions, \eqref{eq:examplefold} takes the form
\begin{equation} \label{eq:examplefoldalpha}
    \dot{x}(t) = \alpha_1 + \alpha_2x(t) + \beta_1 x(t-1) + \beta_2(t)[x(t-1)]^2 + \mathcal{O}([x(t-1)]^3), \quad t \geq 0,
\end{equation}
where the $\mathcal{O}$-terms are $T$-periodic. Here, $\beta_1 \coloneqq D_2g(t,0) \in \mathbb{R}$ as $g$ has autonomous linear part and $\beta_2 \coloneqq (t \mapsto \frac{1}{2}D_2^2g(t,0)) \in C_T(\mathbb{R},\mathbb{R})$ as $g \in C^2(\mathbb{R}^2,\mathbb{R})$ and $T$-periodic in its first argument.

Let us first study \eqref{eq:examplefoldalpha} at the parameter value $\alpha_1 = 0$. Then, $\overline{x} = 0$ is an equilibrium with associated characteristic matrix
\begin{equation*}
    \Delta(z,\alpha_2) = z - \alpha_2 - \beta_1 e^{-z} = -(\alpha_2+\beta_1) + (1+\beta_1)z - \frac{\beta_1}{2}z^2 + \mathcal{O}(z^3).
\end{equation*}
We note that $\lambda = 0$ is a simple root of the map $z \mapsto \Delta(z,\alpha_2)$ whenever
\begin{equation} \label{eq:foldcond}
    \alpha_2 = -\beta_1 \neq 1.
\end{equation}
In the special case $\alpha_2 = -\beta_1 = 1$, the root $\lambda = 0$ has algebraic multiplicity $2$, giving rise generically to a periodically forced Bogdanov-Takens bifurcation. This bifurcation has been further analysed in \cite[Section IX.10]{Diekmann1995} for the autonomous case of \eqref{eq:examplefold}, meaning that $g$ is $t$-independent. Under the assumption of \eqref{eq:foldcond}, we are at a fold bifurcation point with $\sigma_0 = \{0\}$. The null vectors of $\Delta(0,-\beta_1)$ satisfying \eqref{eq:foldnormalization} are given by $q = 1$ and $p = 1/(1+\beta_1)$ so that the eigenfunction $\varphi$ of $A$ reads $\varphi(\theta) = 1$ for all $\theta \in [-1,0]$. Our next aim is to compute the critical normal coefficient $b$ of the periodically forced fold bifurcation in \eqref{eq:examplefoldalpha} at $\alpha_1 = 0$. To evaluate \eqref{eq:criticalcoeffffold}, note that $B$ is of the form 
\begin{equation*}
    B(t;\psi_1(t),\psi_2(t)) = 2 \beta_2(t)[\psi_1(t)(-1) \psi_2(t)(-1)],
\end{equation*}
for all $\psi_{1,2} \in C_T(\mathbb{R},X)$. Using the formula \eqref{eq:criticalcoeffffold} for the critical normal form coefficient, we obtain
\begin{equation} \label{eq:bfoldexample}
    b = \frac{\overline{\beta_2}}{1+\beta_1} , \quad \overline{\beta_2} \coloneqq \frac{1}{T} \int_0^T \beta_2(t) dt,
\end{equation}
where $\overline{\beta_2}$ denotes the average value of $\beta_2$. Note that this formula for $b$ is well-defined due to the assumption \eqref{eq:foldcond}. If $\overline{\beta_2} \neq 0$, \Cref{prop:foldNF} tells us that the periodically forced fold bifurcation is nondegenerate, while a periodically forced cusp bifurcation generically occurs when $\overline{\beta_2} = 0$.

Let us now study \eqref{eq:examplefold} on the punctured fold bifurcation line \eqref{eq:foldcond} in the $(\alpha_2,\beta_1)$-plane, where $\alpha_1$ must be treated as an unfolding parameter. In this setting, \eqref{eq:examplefoldalpha} reads
\begin{equation} \label{eq:examplefoldalpha2}
    \dot{x}(t) = \alpha_1 + \beta_1 (x(t-1) - x(t)) + \beta_2(t)[x(t-1)]^2 + \mathcal{O}([x(t-1)]^3), \quad t \geq 0,
\end{equation}
where we recall that $\beta_1 \neq -1$. If $\overline{\beta_2} \neq 0$, then the fold bifurcation occurring at $\alpha_1 = 0$ is nondegenerate. The normal form analysis from \Cref{subsec:fold} tells us that there exists two $T$-periodic hyperbolic limit cycles $\Gamma_{\pm}$ of opposite stability for $\alpha_1 < 0$ that collide at $\alpha_1 = 0$ and disappear for $\alpha_1 > 0$. This behaviour is also confirmed numerically in \Cref{fig:foldDDE} for the truncated system \eqref{eq:examplefoldalpha2} with $\beta_1 = 1$ and $\beta_2(t) = 1 + \sin(t)$ as $\overline{\beta_2} = 1$ so that the critical normal form coefficient $b$ from \eqref{eq:bfoldexample} does not vanish.

\begin{figure}[ht]
    \centering
    \includegraphics[width=0.95\linewidth]{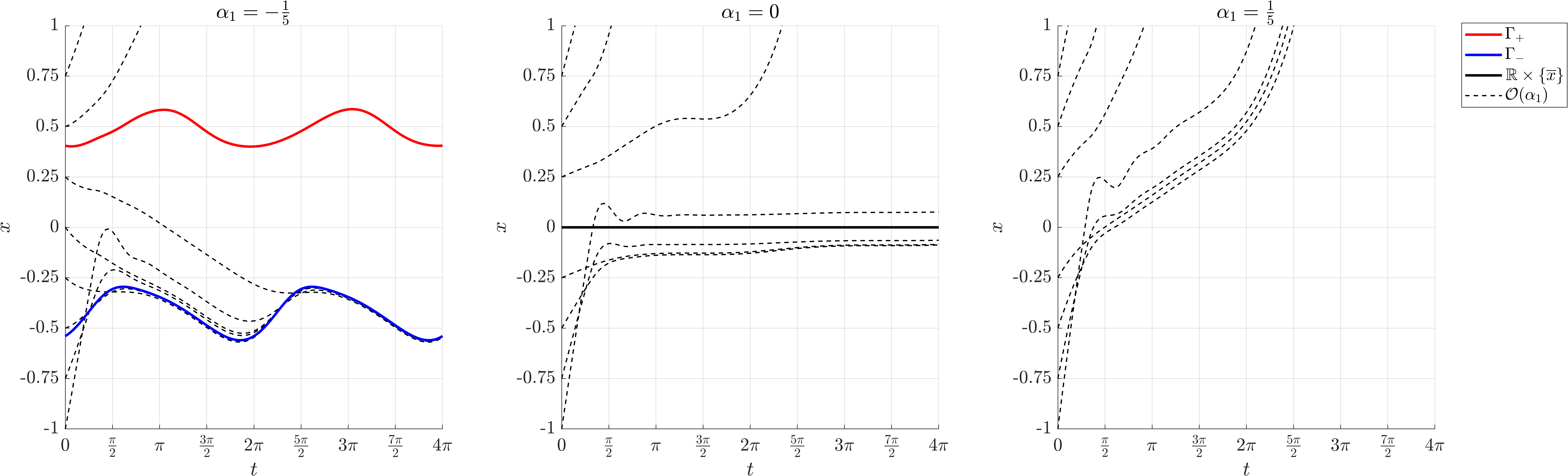}
    \caption{Illustration of a nondegenerate periodically forced fold bifurcation in the truncated scalar DDE \eqref{eq:examplefoldalpha2} with $\beta_1 = 1$ and $\beta_2(t) = 1 + \sin(t)$. Shown are the invariant sets $\Gamma_{\pm}$ ($2\pi$-periodic hyperbolic limit cycles) and $\mathbb{R} \times \{0\}$ (nonhyperbolic equilibrium) together with several forward orbits $\mathcal{O}(\alpha_1)$ that are computed with initial conditions $\varphi \in \{0,\pm\frac{1}{4}, \pm \frac{1}{2}, \pm \frac{3}{4}, \pm 1\} \subseteq X$.}
    \label{fig:foldDDE}
\end{figure}

\subsection{The periodically forced Wright equation} \label{subsec:Wright}

Consider the periodically forced scalar DDE
\begin{equation} \label{eq:wright}
    \dot{x}(t) = ax(t-1)[1+\beta(t)x(t)], \quad t \geq 0,
\end{equation}
where $a \in \mathbb{R}$ is a parameter and $\beta \in C_{T}(\mathbb{R},\mathbb{R})$ is given. This equation can be regarded as a periodically forced version of the well-studied Wright's equation ($\beta = 1$), see \cite{Lessard2010} and the references therein for more information. Clearly, $\overline{x} = 0$ is an equilibrium of \eqref{eq:wright} with associated characteristic matrix
\begin{equation} \label{eq:CharmatrixWright}
    \Delta(z,a) = z-ae^{-z} = -a + (1+a)z -\frac{a}{2}z^2 + \mathcal{O}(z^3).
\end{equation}
Let us now examine the periodically forced fold and Hopf bifurcations in \eqref{eq:wright}.

First, \eqref{eq:CharmatrixWright} has a unique trivial root if and only if $a = 0$. Under these conditions, we are at a fold bifurcation point with $\sigma_0= \{0\}$. The null vectors of $\Delta(0,0)$ satisfying \eqref{eq:foldnormalization} are given by $p = q = 1$ so that the eigenfunction $\varphi$ of $A$ reads $\varphi(\theta) = 1$ for all $\theta \in [-1,0]$. As $B = 0$, the critical normal form coefficient $b$ from \eqref{eq:criticalcoeffffold} vanishes and thus the periodically forced fold bifurcation is degenerate.

Second, \eqref{eq:CharmatrixWright} has a unique pair of simple roots on the imaginary axis if and only if $a = a_N$, where
\begin{equation*}
    a_N \coloneqq (-1)^{N+1} \omega_N, \quad \omega_N \coloneqq \frac{\pi}{2} + N \pi,
\end{equation*}
for some $N \in \mathbb{N}$. Under these conditions, we are at a Hopf bifurcation point with $\sigma_0 = \{ \pm i \omega_N\}$. The null vectors of $\Delta(i \omega_N,a_N)$ satisfying \eqref{eq:hopfnormalization} are given by $q = 1$ and $p = (1-i \omega_N)/(1+\omega_N^2)$ so that the eigenfunction $\varphi$ of $A$ reads $\varphi(\theta) = \exp(i\omega_N \theta)$ for all $\theta \in [-1,0]$. Since a nonresonant Hopf bifurcation requires $\omega_T / \omega_N \notin \mathbb{Q}$, we must assume that $T$ is irrational or not rationally related to $4/(2N +1)$. Our next aim is to compute the first Lyapunov coefficient $l_1$ for \eqref{eq:wright} under the aforementioned nonresonant Hopf bifurcation conditions. To evaluate \eqref{eq:criticalcoeffHopf}, note that $C = 0$ and $B$ is of the form
\begin{equation*}
    B(t;\psi_1(t),\psi_2(t)) = a_N \beta(t)[\psi_1(t)(0) \psi_2(t)(-1) + \psi_1(t)(-1) \psi_2(t)(0)],
\end{equation*}
for all $\psi_{1,2}\in C_T(\mathbb{R},X)$. Using the Fourier expansions from \eqref{eq:v0w0Fourier}, we obtain for $H_{20}$ and $H_{11}$ the expressions
\begin{equation*}
    H_{20}(t)(\theta) = \sum_{m \in \mathbb{Z}}H_{20}^m e^{i(m\omega_T t + (2 \omega_N + m \omega_T)\theta)}, \quad H_{11}(t)(\theta) = \sum_{m \in \mathbb{Z}}H_{11}^m e^{i m\omega_T  (t+\theta)}.
\end{equation*}
Here, $H_{20}^m = J(2 \omega_N +m \omega_T) \beta_m$ and $H_{11}^m = J(m \omega_T) \beta_m$, where $\beta_m$ denote the Fourier coefficients of the function $\beta$, and the map $J : \mathbb{R} \to \mathbb{C}$ is defined by
\begin{equation} \label{eq:J(x)}
    J(x) \coloneqq \frac{2 \omega_N (a_N \sin(x) + x - a_N \cos(x)i)}{a_N^2 + 2a_N x \sin(x) + x^2}.
\end{equation}
Hence, we obtain for all $t \in \mathbb{R}$ the expressions
\begin{align*}
    B(t;\overline{\varphi},H_{20}(t)) &= \omega_N \beta(t) \sum_{m \in \mathbb{Z}} ((-1)^N e^{-im \omega_T} -  i)H_{20}^m e^{im \omega_T t}, \\
    B(t; \varphi,H_{11}(t)) &= -\omega_N \beta(t) \sum_{m \in \mathbb{Z}} ((-1)^N e^{-im \omega_T} -  i)H_{11}^m e^{im \omega_T t},
\end{align*}
meaning that the second component in the pairing of the critical normal form coefficient $c$ from \eqref{eq:criticalcoeffHopf} takes the form
\begin{equation*}
    \omega_N \beta(t) \sum_{m \in \mathbb{Z}}\beta_m J_m((-1)^N e^{-im \omega_T}-i)  e^{im \omega_T t},
\end{equation*}
where $J_m \coloneqq J(2\omega_N + m \omega_T) -2 J(m\omega_T)$ for all $m \in \mathbb{Z}$. Filling everything into \eqref{eq:criticalcoeffHopf} while using the product formula for Fourier series yields eventually
\begin{equation*}
    l_1 = \frac{1}{2(1+\omega_N^2)} \sum_{m \in \mathbb{Z}} \bigg((-1)^N \cos(m\omega_T)(\Re J_m + \omega_N \Im J_m) + (\sin(m\omega_T)+1)(\Im J_m - \omega_N \Re J_m)\bigg)\beta_{-m}.
\end{equation*}

\begin{example} \label{ex:Wright0}
Consider the standard Wright equation $(\beta = 1)$. Then, $\beta_0 = 1$ and $\beta_m = 0$ for $m \neq 0$ so that $J_0 = (4-22(-1)^Ni)/5$. Hence,
\begin{equation} \label{eq:l1wright}
    l_1 = -\frac{1}{5(1+\omega_N^2)}\bigg(9(-1)^{N} + 13 \omega_N\bigg) < 0, \quad \forall N \in \mathbb{N},
\end{equation}
meaning that each Hopf bifurcation is nondegenerate and supercritical. The well-known formula \eqref{eq:criticalcoeffHopfauto} also confirms this result directly. Furthermore, our derivation aligns exactly with existing literature regarding the sign of $l_1$ in Wright's equation, see \cite{Balazs2020,Faria1995,Faria1997,Chow1977}. \hfill $\lozenge$    
\end{example}

\begin{example} \label{ex:Wrightbeta}
Consider the periodically forced Wright equation \eqref{eq:wright} with $\beta(t) = \Omega_1\cos(\Omega_2 t)$ for all $t \in \mathbb{R}$, where $\Omega_{1,2} > 0$ are parameters so that $T = 2\pi/\Omega_2$ and $\omega_T = \Omega_2$. By inspecting the roots of $z \mapsto \Delta(z,a)$, the principle of linearized stability implies that the zero solution of \eqref{eq:wright} is asymptotically stable precisely for $a \in \left(-\tfrac{\pi}{2},0\right)$, see also ``$0$-(Un)stable Region'' in \Cref{fig:BifDiagram}. Furthermore, recall that $a = 0$ corresponds to a degenerate fold bifurcation for all $\Omega_{1,2}$, which is also illustrated in \Cref{fig:BifDiagram}. 

\begin{figure}[ht]
    \centering
    \includegraphics[width=0.95\linewidth]{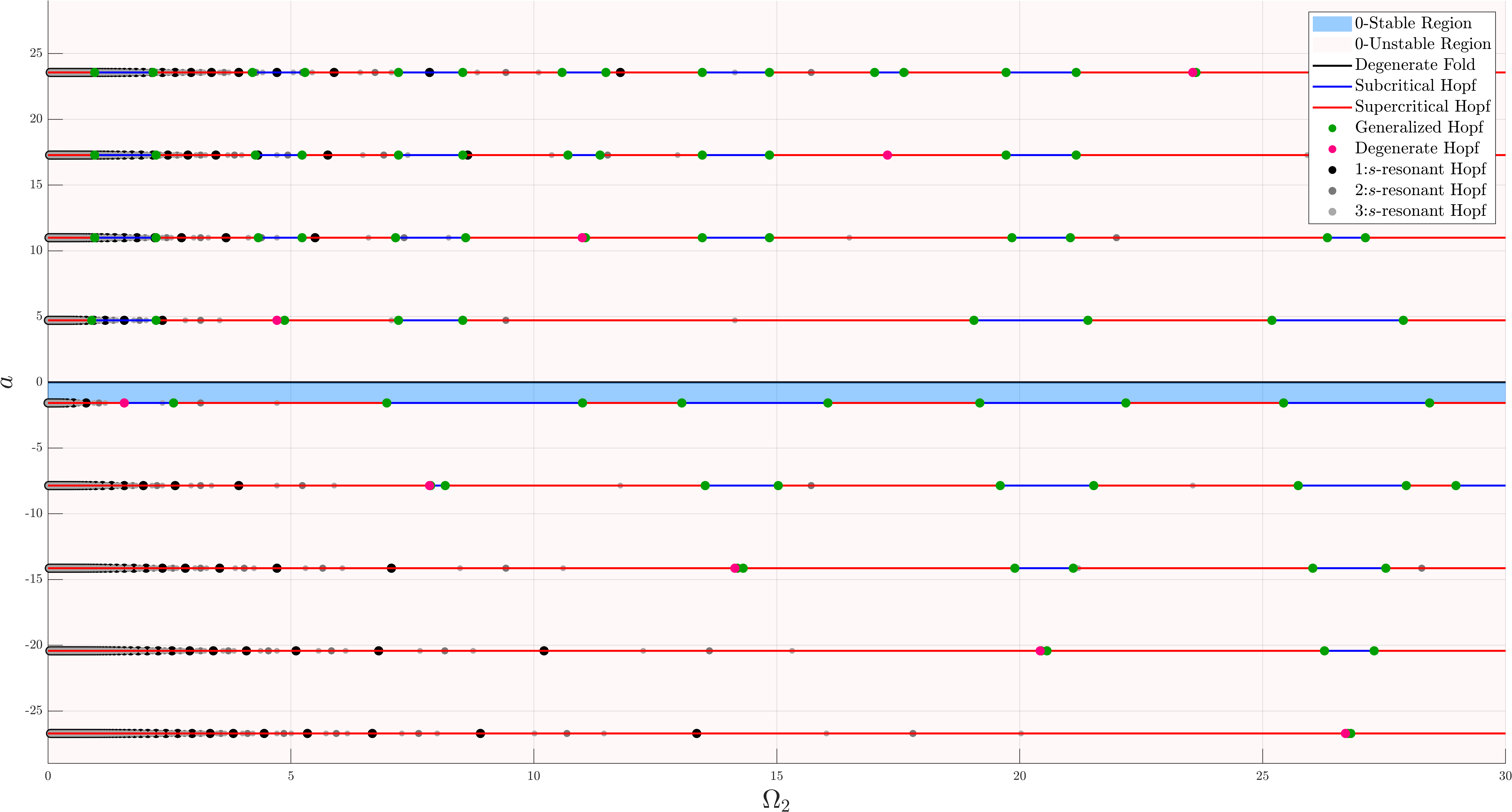}
    \caption{Bifurcation diagram of the nonlinearly periodically forced Wright equation \eqref{eq:wright} with forcing function $\beta(t) = \cos(\Omega_2t)$ and unfolding parameters $(\Omega_2,a)$. Horizontal lines with $a \neq 0$ correspond to different Hopf branches.}
    \label{fig:BifDiagram}
\end{figure}

With this choice of $\beta$, note that \eqref{eq:wright} undergoes a $r\!:\!s$-resonant Hopf bifurcation on the $N$th Hopf branch if $\Omega_2 = \omega_N\frac{r}{s}$. The strongly resonant Hopf points $(r\in \{1,2,3\} \text{ and } s \in \mathbb{N}_0)$ are displayed in \Cref{fig:BifDiagram}. On the other hand, if $\Omega_2 \notin \omega_N \mathbb{Q}$ on the $N$th Hopf branch, then a nonresonant Hopf bifurcation occurs so that we can compute the first Lyapunov coefficient $l_1$ using the results from \Cref{subsec:hopf}. In this setting, it is clear that $\beta_{\pm 1} = \Omega_1/2$ and $\beta_m = 0$ for $m \neq \pm1$ so that $J_{\pm1} = J(2\omega_N \pm \Omega_2) - 2 J(\pm\Omega_2)$. Hence, the first Lyapunov coefficient on the $N$th Hopf branch reads
\begin{align}
\begin{split} \label{eq:l1PF}
    l_1 &= \frac{\Omega_1}{4(1+\omega_N^2)} \bigg((-1)^N \cos(\Omega_2)[\Re (J_{-1} + J_1) + \omega_N \Im (J_{-1} + J_1)] \\
    &+ (1-\sin(\Omega_2))[\Im J_{-1} - \omega_N \Re J_{-1}] + (1+\sin(\Omega_2))[\Im J_{1} - \omega_N \Re J_{1}]\bigg).
\end{split}
\end{align}
Since $\Omega_1 > 0$, the sign of $l_1$ does not depend on the specific value of $\Omega_1$. We therefore fix $\Omega_1 = 1$ from now on and consider $l_1$ from \eqref{eq:l1PF} as a function of $\Omega_2 > 0$ along a fixed Hopf branch. A straightforward calculation shows that $l_1(\Omega_2) \to l_1^0$ as $\Omega_2 \downarrow 0$, where $l_1^0$ denotes the first Lyapunov coefficient of the standard Wright equation $(\beta = 1)$ given in \eqref{eq:l1wright}. By continuity of $l_1$, we conclude that each nonresonant Hopf bifurcation on each Hopf branch is supercritical for sufficiently small $\Omega_2 > 0$ not belonging to $\omega_N\mathbb{Q}$. This behaviour is also confirmed (numerically) in \Cref{fig:BifDiagram}. 

For larger values of $\Omega_2 >0$ along a fixed Hopf branch, $l_1$ may change sign. This can occur in two distinct ways. First, if $l_1$ passes through zero, the system \eqref{eq:wright} undergoes a generalized Hopf bifurcation, a scenario that has been extensively analysed for autonomous DDEs in \cite{Bosschaert2020,Delmeire2025,Janssens2010}. These codim 2 bifurcation points are computed numerically and illustrated in \Cref{fig:BifDiagram}. Second, if the sign change arises due to a vertical asymptote of $l_1$, then \eqref{eq:wright} undergoes a \emph{degenerate Hopf} bifurcation, indicating a breakdown of the normal form (approximation) from \Cref{prop:HopfPF}. These (codim 2) bifurcation points can be computed explicitly and are illustrated in \Cref{fig:BifDiagram}. Indeed, for $l_1$ to exhibit a vertical asymptote, it is necessary that $w(x) \coloneqq a_N^2 + 2 a_N x \sin(x) + x^2$ vanishes at one of the $x \in \{\pm \Omega_2, 2\omega_N \pm \Omega_2\}$. Substituting these specific values into $w$ reveals that this occurs only on the $N$th Hopf branch at $\Omega_2 = \omega_N$, indicating a $1\!:\!1$-resonant Hopf bifurcation. However, in contrast to the previously computed $1\!:\!s$-resonant Hopf points for all $s \in \mathbb{N}$, the sign of $l_1$ changes at this point, which is not a generic scenario. 
\hfill $\lozenge$
\end{example}

We observed in \Cref{ex:Wrightbeta} that the periodically forced Wright equation reduces to the classical Wright equation from \Cref{ex:Wright0} as $\Omega_2 \downarrow 0$, while its bifurcation structure is preserved. In both cases, recall that a Hopf bifurcation occurs at the parameter value $a = a_N$ for every $N \in \mathbb{N}$. Since the trivial solution is asymptotically stable for $(\Omega_2,a) \in [0,\infty) \times \left(-\tfrac{\pi}{2},0\right)$, the supercritical nature of a nonresonant Hopf bifurcation can also be verified numerically by computing forward orbits of \eqref{eq:wright} for nonresonant values of $\Omega_2$, with $a$ chosen close to $-\tfrac{\pi}{2}$. When crossing a supercritical nonresonant Hopf bifurcation, recall from \Cref{subsec:hopf} that a nonautonomous stable invariant set exists for parameter values slightly below $-\tfrac{\pi}{2}$. At $\Omega_2 = 0$ this invariant set is an exact cylinder, whereas for $\Omega_2 > 0$ it deforms into a distorted cylinder. This theoretical result is also confirmed numerically in \Cref{fig:WrightHopf}, where $C(\Omega_2,a)$ denotes the projection of this cylindrical surface from $\mathbb{R} \times X$ onto the $(t,x(t),x(t-1))$-space.

\begin{figure}[ht]
    \centering
    \includegraphics[width=0.95\linewidth]{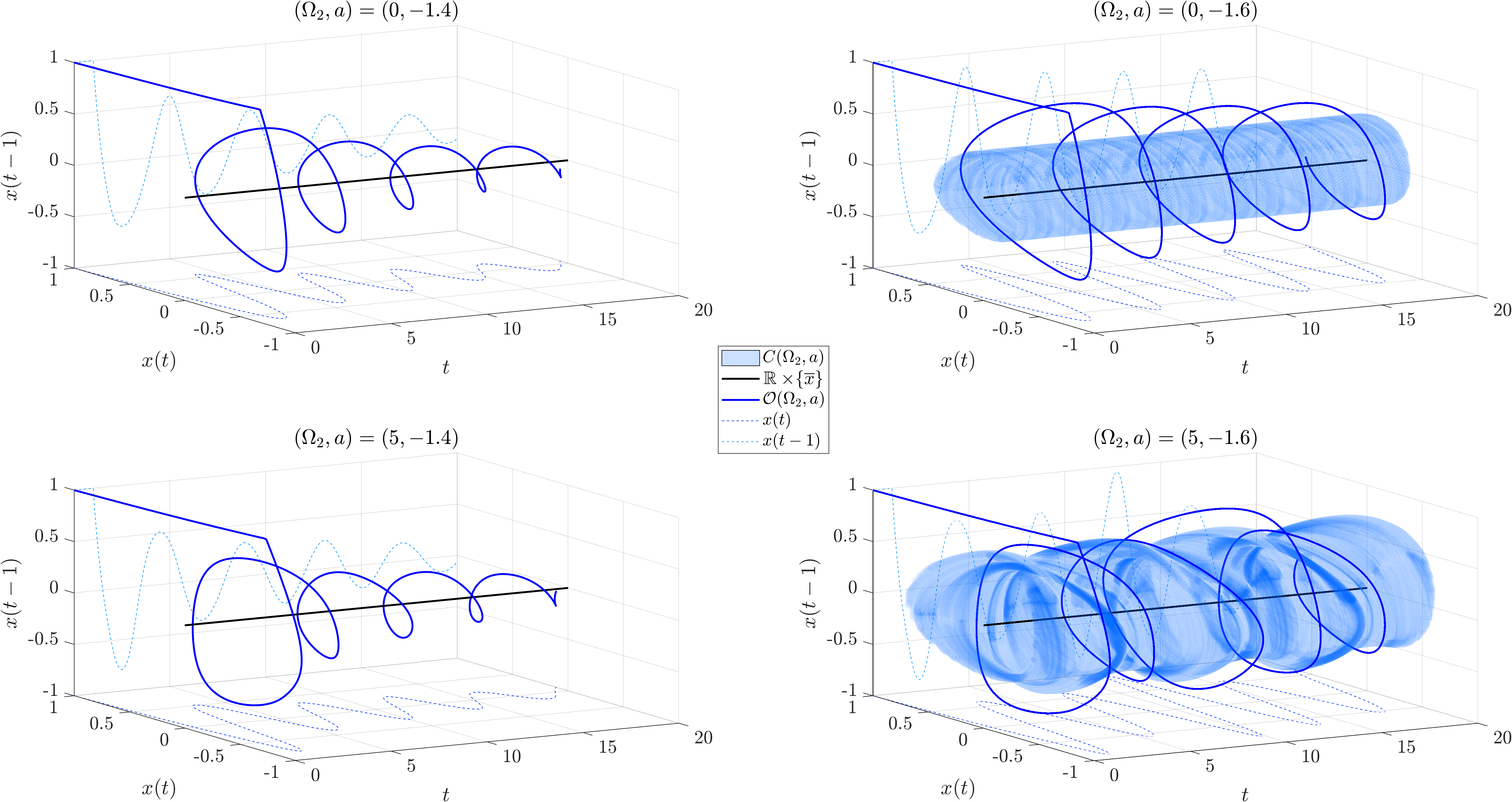}
    \caption{Several forward orbits $\mathcal{O}(\Omega_2,a)$ computed using the initial condition $\varphi=1 \in X$, together with the invariant sets $\mathbb{R} \times \{\overline{x}\}$ and $C(\Omega_2,a)$, illustrated for various parameter values $(\Omega_2,a)$ of the periodically forced Wright equation \eqref{eq:wright} with forcing function $\beta(t)=\cos(\Omega_2 t)$, shown in $(t,x(t),x(t-1))$ space.}
    \label{fig:WrightHopf}
\end{figure}

\section{Conclusion and outlook} \label{sec:outlook}
We developed a general framework for analysing bifurcations of equilibria in nonlinearly periodically forced DDEs. Our approach combines the construction of periodic center manifolds, the derivation of periodically forced normal forms, and a normalization method for computing (critical) normal form coefficients. We provided a rigorous center manifold theorem within the sun-star calculus framework and demonstrated that the resulting normal forms are structurally identical to those for periodically forced ODEs. In particular, we derived explicit computational formulas for the critical normal form coefficients of the fold and nonresonant Hopf bifurcation.

A natural next step is to extend the present framework to bifurcations in parameter-dependent nonlinearly periodically forced DDEs. This requires the development of a parameter-dependent center manifold and normal form theory in the spirit of \cite[Section 3.5]{Bosschaert2020} and \cite[Theorem 3.5.2]{Haragus2011}. This approach enables the derivation of explicit computational formulas for the critical and parameter-dependent normal form coefficients associated with all codim 1 and 2 bifurcations of equilibria \cite{Janssens2010,Bosschaert2020,Delmeire2025}. In particular, it facilitates the computation of normal form coefficients for the (strongly) resonant Hopf bifurcations, which do not appear in autonomous DDEs. For example, this would enable us to extend the bifurcation diagram presented in \Cref{fig:BifDiagram} massively. After completing this task, it would be valuable to derive predictors for (secondary) codim 1 curves (of limit cycles) emanating from codim 2 points, as discussed in \cite{Bosschaert2025,Bosschaert2024a} for autonomous DDEs. This would enable the study of bifurcations of such limit cycles, requiring a generalization of the results in \cite{Article1,Chicone1997,Article2,Bosschaert2025} towards nonlinearly periodically forced DDEs. We expect interesting types of dynamics to arise here, as the period of the bifurcating cycle may resonate with that of the nonlinearity. The ultimate objective is to incorporate all these formulas for both periodically forced ODEs and DDEs into software packages such as \verb|MatCont| \cite{Dhooge2003,Dhooge2008}, \verb|DDE-BifTool| \cite{Engelborghs2002,Sieber2014} or \verb|BifurcationKit| \cite{Veltz2020}.

The techniques developed in this work are broadly applicable and naturally invite further generali\-zations. Since the center manifold theorem (\Cref{thm:LCMT}) is formulated within the abstract framework of dual perturbation theory, it is plausible that the periodically forced normal form result (\Cref{thm:periodicnormalform}) remains valid in more and diverse general settings. However, depending on the type of nonlinearly periodically forced evolution equation under consideration, one may need to relax the $C^k$-smoothness assumption inherent to classical DDEs. For instance, one could consider piecewise $C^k$-smoothness in the context of impulsive systems \cite{Church2021} or Sobolev-type regularity such as $W^{k,p}$-smoothness for DDEs on $L^p$-spaces \cite{Batkai2005} or partial delay differential equations \cite{Wu1996}. In the same spirit, we anticipate that the normalization method proposed here can be extended to a broader class of nonlinearly periodically forced systems. Notably, the derived critical normal form coefficients \eqref{eq:criticalcoeffffold} and \eqref{eq:criticalcoeffHopf} are expressed in an $L^2$-style formulation, as evidenced by the bilinear pairing $\langle \cdot , \cdot \rangle_T$ introduced in \eqref{eq:pairingT2}. This structure suggests that our formulas remain meaningful under weaker (regularity) assumptions, such as those arising in the analysis of renewal equations, neutral equations, delay equations with state-dependent or infinite delay, abstract delay equations, partial delay differential equations, or even mixed functional differential equations.

\section*{Acknowledgements}
The authors thank Prof. A. J. Roberts (University of Adelaide) for insightful comments on an earlier version of this manuscript.

\bibliographystyle{siamplain}
\bibliography{references}

\end{sloppypar}
\end{document}